\documentclass[smallextended,numbook,runningheads]{svjour3}     
\smartqed  
\usepackage{amssymb,amsfonts,amsmath}
\usepackage[mathscr]{eucal}
\usepackage{stmaryrd} 
\usepackage{bm}
\usepackage[usenames,dvipsnames]{xcolor}  
\usepackage{algorithm,algpseudocode} 
\usepackage{xparse}
\usepackage{subfig}
\usepackage{booktabs} 
\usepackage{graphicx} 
\usepackage{siunitx} 
\usepackage{tablefootnote} 
\usepackage{hyperref}
\usepackage[nameinlink]{cleveref} 

\DeclareMathOperator{\Trace}{Trace}
\DeclareMathOperator{\TR}{TR}
\DeclareMathOperator{\rank}{rank}
\DeclareMathOperator{\OPT}{OPT}
\DeclareMathOperator{\range}{range}

\NewDocumentCommand \tensor {O{}m} {\boldsymbol{#1\mathscr{\MakeUppercase{#2}}}} 
\newcommand{\mat}[1]{\mathbf{#1}}
\newcommand{\vect}[1]{\bm{#1}}

\newcommand{\bb}[1]{\mathbb{#1}}
\newcommand{\bigO}[1]{\mathcal{O}\left( #1 \right)}
\newcommand{\defeq}{\stackrel{\text{\tiny \textnormal{def}}}{=}}  

\makeatletter
\newenvironment{breakablealgorithm}
  {
     \refstepcounter{algorithm}
     \hrule height.8pt depth0pt \kern2pt
     \renewcommand{\caption}[2][\relax]{
       {\raggedright\textbf{\ALG@name~\thealgorithm} ##2\par}%
       \ifx\relax##1\relax 
         \addcontentsline{loa}{algorithm}{\protect\numberline{\thealgorithm}##2}%
       \else 
         \addcontentsline{loa}{algorithm}{\protect\numberline{\thealgorithm}##1}%
       \fi
       \kern2pt\hrule\kern2pt
     }
  }{
     \kern2pt\hrule\relax
  }
\makeatother

\journalname{}
\begin{document}
\begin{sloppypar}

\title{Tracking Tensor Ring Decompositions of Streaming Tensors\thanks{The work is supported by the National Natural Science Foundation of China (No. 11671060) and the
Natural Science Foundation of Chongqing, China (No. cstc2019jcyj-msxmX0267)
}}


\author{Yajie Yu         \and
        Hanyu Li 
}


\institute{Yajie Yu \at
              College of Mathematics and Statistics, Chongqing University, Chongqing, 401331, P.R. China; \\
              \email{zqyu@cqu.edu.cn}           
           \and
           Hanyu Li \at
              College of Mathematics and Statistics, Chongqing University, Chongqing, 401331, P.R. China; \\
              \email{lihy.hy@gmail.com or hyli@cqu.edu.cn}
}

\date{Received: date / Accepted: date}

\maketitle

\begin{abstract}
Tensor ring (TR) decomposition is an efficient approach to discover the hidden low-rank patterns for higher-order tensors, and streaming tensors are becoming highly prevalent in real-world applications. 
In this paper, we investigate how to track TR decompositions of streaming tensors. 
An efficient 
algorithm 
is first proposed. 
Then, based on this algorithm and randomized techniques, we present a randomized streaming TR decomposition. The proposed algorithms make full use of the structure of TR decomposition, and the randomized version can allow any sketching type.
Theoretical results on sketch size are provided. In addition, the complexity analyses for the obtained algorithms are also given. 
We compare our proposals with the existing batch methods using both real and synthetic data. 
Numerical results show that they have better performance in computing time with maintaining similar accuracy.
\end{abstract}

\section{Introduction}
\label{sec:introduction}
Tensor ring (TR) decomposition \cite{zhao2016TensorRing} is an important tool for higher-order data analysis. It decomposes an $N$th-order tensor into a cyclic interconnection of $N$ 3rd-order tensors, and hence has the advantage of circular dimensional permutation invariance.
Specifically, for a tensor $\tensor{X} \in \bb{R}^{I_1 \times I_2 \times \cdots \times I_N}$, it has the TR format as follows
\begin{align*}
	\tensor{X}(i_1, \cdots, i_N)&={\Trace} \left(\mat{G}_1(i_1) \mat{G}_2(i_2) \cdots \mat{G}_N(i_N)\right)
	={\Trace} \left(\prod_{n=1}^N \mat{G}_n(i_n)\right),
\end{align*}
where $\mat{G}_n(i_n) = \tensor{G}_n(:,i_n,:) \in \bb{R}^{R_n \times R_{n+1}}$ is the $i_n$-th \emph{lateral slice} of the \emph{core tensor (TR-core)} $\tensor{G}_n \in \bb{R}^{R_n \times I_n \times R_{n+1}}$. 
Note that a \emph{slice} is a 2nd-order section, i.e., a matrix, of a tensor obtained by fixing all the tensor indices but two. The sizes of TR-cores, i.e., $R_k$ with $k=1,\cdots,N$ and $ R_{N+1}=R_{1} $, are called \emph{TR-ranks}. 
Additionally, we use the notation $\TR \left( \{\tensor{G}_n\}_{n=1}^N \right)$ to denote the TR decomposition of a tensor.

In contrast to the two most popular tensor decompositions, i.e., CANDECOMP-PARAFAC (CP) and Tucker decompositions \cite{kolda2009TensorDecompositions}, TR decomposition avoids the NP-hard problem of computing the CP-rank and the curse of dimensionality due to the core tensor in Tucker decomposition. 
These advantages stem from the algorithms for finding TR-ranks being stable and the number of parameters of TR decomposition scaling linearly with the tensor 
order $N$.
Furthermore, it is also feasible and convenient to directly implement some algebra operations in TR format \cite{zhao2016TensorRing}, e.g., addition, dot product, norm, matrix-by-vector, etc., which is conducive to significantly enhancing the computational efficiency. 
Therefore, TR decomposition has already been utilized effectively in scientific computing to tackle intractable higher-dimensional and higher-order problems. 
This makes the problem of fitting $\TR \left( \{\tensor{G}_n\}_{n=1}^N \right)$ to a tensor $\tensor{X}$ be increasingly important.

The above fitting problem can be written as the following minimization problem:
\begin{equation}
	\label{eq:trmin}
	\mathop{\min}_{\tensor{G}_1, \cdots, \tensor{G}_N} \left\| \TR \left( \{\tensor{G}_n\}_{n=1}^N \right) - \tensor{X} \right\|_F,
\end{equation} 
where $\| \cdot \|_F$ denotes the Frobenius norm of a matrix or tensor. 
One 
standard computational method for this problem is to prescribe the fixed TR-ranks first, and then to determine the decomposition via alternating least squares (ALS). It is usually written as TR-ALS. 
Another classical method is to prescribe a fixed target accuracy first, and then to compute the decomposition via singular value decomposition. 
See \cite{zhao2016TensorRing,mickelin2020AlgorithmsComputing} for the details on these two methods. Moreover, with the rapid emergence of large-scale problems, 
the above two methods have been extended to randomized versions \cite{yuan2019RandomizedTensor,ahmadi-asl2020RandomizedAlgorithms,malik2021SamplingBasedMethod,malik2022MoreEfficient,yu2022PracticalSketchingBased}. 
In addition, many other algorithms have also been proposed and developed for TR decomposition; see, e.g., \cite{espig2012NoteTensor,ahmadi-asl2021CrossTensor,yuan2018HigherdimensionTensor}. 
However, all of these works are for the case  
that the whole tensor data $\tensor{X}$ is available and static. 

As we know, 
in many practical applications, there only fragmentary data sets are initially available, with new data sets becoming available at the next time step or appearing continuously over time. Live video broadcasts, surveillance videos, network flow, and social media data are examples.
Such tensors are called streaming tensors or incremental/online tensors \cite{sun2008IncrementalTensor}.
Developing streaming algorithms for TR decompositions, i.e., tracking TR decompositions, of such streaming tensors is both fascinating and necessary. This is because when the initial decomposition is already known, it is more expedient to update the streaming decomposition than to recalculate the entire decomposition.
In previous research, some streaming methods have been successively presented for some other tensor decompositions; see e.g.,  \cite{zhou2016AcceleratingOnline,ma2018RandomizedOnline,zeng2021IncrementalCP} for CP decomposition, \cite{sun2006StreamsGraphs,sun2008IncrementalTensor,chachlakis2021DynamicL1Norm,xiao2018EOTDEfficient,sun2020LowRankTucker} for Tucker decomposition,  and \cite{liu2021IncrementalTensorTrain,thanh2021AdaptiveAlgorithms,kressner2022StreamingTensor} for tensor train (TT) decomposition \cite{oseledets2011TensorTrainDecomposition}. A more comprehensive and detailed overview can be found in \cite{thanh2022ContemporaryComprehensive}.
However, streaming algorithms related to (instead of aiming at) TR decomposition have only been studied in several papers \cite{he2022PatchTrackingbased,yu2022OnlineSubspace,huang2022MultiAspectStreaming}. 
Specifically, He and Atia \cite{he2022PatchTrackingbased} developed a patch-tracking-based streaming TR completion framework for visual data recovery and devised 
a streaming 
algorithm 
that can update the latent TR-cores and complete the missing entries of patch tensors. 
Yu et al. \cite{yu2022OnlineSubspace} proposed an online TR subspace learning and imputation model by formulating exponentially weighted least squares with Frobenius norm regularization of TR-cores. The alternating recursive least squares and stochastic gradient algorithms were employed to solve the proposed model.
Huang et al. \cite{huang2022MultiAspectStreaming} provided a multi-aspect streaming TR completion method. 
Whereas, all of these works didn't fully consider the special structure of TR decomposition. 
From \cite{yu2022PracticalSketchingBased,yu2022PracticalAlternating}, it is shown that exploring the structure 
can significantly improve the efficiency of the related algorithms.

Therefore, in this paper, we focus on developing efficient ALS-based streaming algorithms for TR decomposition via making full use of its structure. 
Specifically,
inspired by the work on CP decomposition in \cite{zhou2016AcceleratingOnline}, we first propose an 
efficient streaming algorithm that can incrementally track TR decompositions of streaming tensors with any order. 
Then, motivated by the ideas in \cite{ma2018RandomizedOnline,malik2021SamplingBasedMethod,yu2022PracticalSketchingBased,yu2022PracticalAlternating}, we derive a 
randomized streaming TR decomposition 
to deal with streaming large-scale tensors. 
Three randomized strategies, i.e., uniform sampling, leverage-based sampling, and Kronecker
sub-sampled randomized Fourier transform (KSRFT), are used to 
reduce the dimension of the coefficient and unfolding matrices in the ALS subproblems, which makes the computing time and memory usage be reduced greatly. 
Moreover, these strategies can also avoid forming the full coefficient and sketching  matrices and implementing matrix multiplication between large matrices.

The rest of this paper is organized as follows. 
\Cref{sec:preliminaries} first gives some tensor notations and basic operations, and then briefly reviews the algorithms for TR decomposition. 
In \Cref{sec:proposed}, we present our streaming TR decomposition and 
its randomized variant as well as three sketching techniques.
The evaluation of the computational performance of the proposed algorithms is reported in \Cref{sec:experiments}. 
Finally, \Cref{sec:conclusion} makes a conclusion and outlines some further directions. The missed proofs and the specific algorithms based on different sketches are given in \Cref{SUPP:proof} and \Cref{SUPP:alg}, respectively.

\section{Preliminaries and Related Works}
\label{sec:preliminaries}

For convenience on the following presentment, 
we denote $[I] \defeq \{ 1, \cdots, I \}$ for a positive integer $I$, and set  $\overline{i_1 i_2 \cdots i_N} \defeq 1 + \sum_{n=1}^N(i_n-1)\prod_{j=1}^{n-1}I_j$ for the indices $i_1 \in [I_1], \cdots, i_N \in [I_N]$. 

\begin{definition}
    Three unfolding matrices of a tensor $\tensor{X} \in \bb{R}^{I_1 \times I_2 \cdots \times I_N}$ are defined element-wise:
    \begin{align*}
        &\text{Classical Mode-$n$ Unfolding:}&\mat{X}_{(n)}(i_n, \overline{i_1 \cdots i_{n-1} i_{n+1} \cdots i_N})=\tensor{X}(i_1, \cdots, i_N),  \\
        &\text{Mode-$n$ Unfolding:}&\mat{X}_{[n]}(i_n, \overline{i_{n+1} \cdots i_N i_1 \cdots i_{n-1}})=\tensor{X}(i_1, \cdots, i_N),  \\
        &\text{$n$ Unfolding:}&\mat{X}_{<n>}(\overline{i_1, \cdots, i_n}, \overline{i_{n+1} \cdots i_N})=\tensor{X}(i_1, \cdots, i_N), 
    \end{align*}
    which are of size $I_n \times \prod_{j \ne n} I_j$, $I_n \times \prod_{j \ne n} I_j$, and $\prod_{j=1}^{n} I_j \times \prod_{j=n+1}^{N} I_j$, respectively.
\end{definition}




\begin{definition}[TTM]
    The \textbf{tensor-times-matrix (TTM) multiplication} of a tensor $\tensor{X} \in \bb{R}^{I_1 \times I_2 \cdots \times I_N}$ and a matrix $\mat{U} \in \bb{R}^{J \times I_n}$ 
    is a tensor of size $I_1 \times \cdots \times I_{n-1} \times J \times I_{n+1} \times \cdots \times I_N$ denoted by $\tensor{X} \times_n \mat{U}$ and defined element-wise via
	\begin{equation*}
	(\tensor{X} \times_n \mat{U})(i_1, \cdots, i_{n-1}, j, i_{n+1}, \cdots, i_N) = \sum_{i_n = 1}^{I_n} \tensor{X}(i_1, \cdots, i_n, \cdots, i_N) \mat{U}(j, i_n).
	\end{equation*}
\end{definition}

Multiplying an $N$th-order tensor by multiple matrices on distinct modes is known as \emph{Multi-TTM}. 
In particular, multiplying an $N$th-order tensor by the matrices $\mat{U}_j$ with $j=1,\cdots, N$ in each mode implies $\tensor{Y} = \tensor{X} \times_1 \mat{U}_1 \times_2 \mat{U}_2 \cdots \times_N \mat{U}_N$. Its mode-$n$ unfolding can be presented as follows:

\begin{align}
\label{prop:ttm}
    \mat{Y}_{[n]} = \mat{U}_n \mat{X}_{[n]} \left( \mat{U}_{n-1} \otimes \cdots \otimes \mat{U}_{1} \otimes \mat{U}_{N} \otimes \cdots \otimes \mat{U}_{n+1} \right)^\intercal.
\end{align}

We now detail the TR-ALS mentioned in \Cref{sec:introduction}, which is a popular algorithm for TR decomposition. 
To achieve this, we need the following definition.
\begin{definition}
\label{def:subchain_tensor}
	Let $\tensor{X} = \TR \left( \{\tensor{G}_n\}_{n=1}^N \right) \in \bb{R}^{I_1 \times I_2 \cdots \times I_N}$. The \textbf{subchain tensor} $\tensor{G}^{\ne n} \in \bb{R}^{R_{n+1} \times \prod_{j \ne n} I_j \times R_n}$ is the merging of all TR-cores except the $n$-th one and can be written slice-wise via
	\begin{equation*}
	\mat{G}^{\ne n}(\overline{i_{n+1} \cdots i_N i_1 \cdots i_{n-1}})=\prod_{j=n+1}^{N} \mat{G}_j(i_j) \prod_{j=1}^{n-1} \mat{G}_j(i_j).
	\end{equation*}
\end{definition}

Thus, according to Theorem 3.5 in \cite{zhao2016TensorRing}, the objective in \eqref{eq:trmin} can be rewritten as the following $N$ subproblems
\begin{equation}
\label{eq:tr_als}
\mat{G}_{n(2)} = \mathop{\arg\min}_{\mat{G}_{n(2)}} \frac{1}{2}\left\|\mat{G}_{[2]}^{\ne n} \mat{G}_{n(2)}^\intercal-\mat{X}_{[n]}^\intercal \right\|_F,\ n=1,\cdots, N.
\end{equation}
The so-called TR-ALS is a method that keeps all TR-cores fixed except the $n$-th one and finds the solution to the LS problem \eqref{eq:tr_als} with respect to it. We summarize the method in \Cref{alg:tr_als}.

\begin{algorithm}
\caption{TR-ALS \cite{zhao2016TensorRing}}
\label{alg:tr_als}
	\textbf{Input:} $\tensor{X} \in \bb{R}^{I_1 \times \cdots \times I_N}$, TR-ranks $R_1, \cdots, R_N$
	
	\textbf{Output:} TR-cores $\{ \tensor{G}_n \in \bb{R}^{R_n \times I_n \times R_{n+1}} \}_{n=1}^N$
	\begin{algorithmic}[1]\footnotesize
		\State Initialize TR-cores $\tensor{G}_1, \cdots, \tensor{G}_N$ \label{line:als_init}
		\Repeat
		\For{$n = 1, \cdots, N$}
		\State Compute $\mat{G}_{[2]}^{\ne n}$ from TR-cores \label{line:als_subchain}
		\State Update $\tensor{G}_n = \mathop{\arg\min}_{\tensor{G}_{n}} \left\| \mat{G}_{[2]}^{\ne n} \mat{G}_{n(2)}^\intercal - \mat{X}_{[n]}^\intercal \right\|_F$ \label{line:als_ls}
		\EndFor
		\Until{termination criteria met}
	\end{algorithmic}
\end{algorithm}

However, TR-ALS does not fully utilize the structure of the coefficient matrix $\mat{G}_{[2]}^{\ne n}$. Yu and Li \cite{yu2022PracticalAlternating} fixed this issue recently and proposed a more efficient algorithm called TR-ALS-NE, which is the basis of our first algorithm in the present paper. 
We first list the required definitions and property before detailing the algorithm in \Cref{alg:tr_als_ne}.

\begin{definition}[Outer Product]
\label{def:outer_product}
	The \textbf{outer product} of two tensors $\tensor{A} \in \bb{R}^{I_1 \times \cdots \times I_N}$ and $\tensor{B} \in \bb{R}^{J_1 \times \cdots \times J_M}$ is a tensor of size $I_1 \times \cdots \times I_N \times J_1 \times \cdots \times J_M$ denoted by $\tensor{A} \circ \tensor{B}$ and defined element-wise via
	\begin{equation*}
		(\tensor{A} \circ \tensor{B})(i_1, \cdots, i_N, j_1, \cdots, j_M) = \tensor{A}(i_1, \cdots, i_N) \tensor{B}(j_1, \cdots, j_M). 
	\end{equation*}
\end{definition}

\begin{definition}[General Contracted Tensor Product]\label{def:ct_product}
	The \textbf{general contracted tensor product} of two tensors $\tensor{A} \in \bb{R}^{I_1 \times J \times R_1 \times K}$ and $\tensor{B} \in \bb{R}^{J \times I_2 \times K \times R_2}$ is a tensor of size $I_1 \times I_2 \times R_1 \times R_2$ denoted by $\tensor{A} \times_{2,4}^{1,3} \tensor{B}$ and defined element-wise via
		\begin{equation*}
		(\tensor{A} \times_{2,4}^{1,3} \tensor{B})(i_1, i_2, r_1, r_2) = \sum_{j,k} \tensor{A}(i_1, j, r_1, k) \tensor{B}(j, i_2, k,r_2).
	\end{equation*}
\end{definition}

\begin{definition}[Subchain Product \cite{yu2022PracticalSketchingBased}] 
\label{def:subchain_product}
	The \textbf{mode-2 subchain product} of two tensors $\tensor{A} \in \bb{R}^{I_1 \times J_1 \times K}$ and $\tensor{B} \in \bb{R}^{K \times J_2 \times I_2}$ is a tensor of size $I_1 \times J_1 J_2 \times I_2$ denoted by $\tensor{A} \boxtimes_2 \tensor{B}$ and  defined as 
	\begin{equation*}
	(\tensor{A} \boxtimes_2 \tensor{B})(\overline{j_1 j_2}) = \mat{A}(j_1)\mat{B}(j_2).
	\end{equation*}
\end{definition}

\begin{proposition}\cite{yu2022PracticalAlternating}
\label{prop:subchain_gram}
	Let $\tensor{A} \in \bb{R}^{I_1 \times J \times K_1}$, $\tensor{B} \in \bb{R}^{K_1 \times R \times L_1}$, $\tensor{C} \in \bb{R}^{I_2 \times J \times K_2}$ and $\tensor{D} \in \bb{R}^{K_2 \times R \times L_2}$ be 3rd-order tensors. Then
	\begin{equation*}
		(\tensor{A} \boxtimes_2 \tensor{B})_{[2]}^\intercal (\tensor{C} \boxtimes_2 \tensor{D})_{[2]} = \left( (\sum_{r=1}^{R} \mat{B}(r)^\intercal \circ \mat{D}(r)^\intercal) \times_{2,4}^{1,3} (\sum_{j=1}^{J} \mat{A}(j)^\intercal \circ \mat{C}(j)^\intercal) \right)_{<2>}.
	\end{equation*}
\end{proposition}

\begin{algorithm}
\caption{TR-ALS-NE \cite{yu2022PracticalAlternating}}
\label{alg:tr_als_ne}
	\textbf{Input:} $\tensor{X} \in \bb{R}^{I_1 \times \cdots \times I_N}$, TR-ranks $R_1, \cdots, R_N$
	
	\textbf{Output:} TR-cores $\{ \tensor{G}_n \in \bb{R}^{R_n \times I_n \times R_{n+1}} \}_{n=1}^N$
	\begin{algorithmic}[1]
		\State Initialize TR-cores $\tensor{G}_{1}, \cdots, \tensor{G}_{N}$
		\State Compute $\tensor{Z}_{1} = \sum_{i_1 = 1}^{I_1} \mat{G}_{1}(i_1)^\intercal \circ \mat{G}_{1}(i_1)^\intercal, \cdots, \tensor{Z}_{N} = \sum_{i_N = 1}^{I_N} \mat{G}_{N}(i_N)^\intercal \circ \mat{G}_{N}(i_N)^\intercal$ 
		\Repeat
		\For{$n = 1,\cdots,N$}
		\State $\tensor{H}^{\ne n} \leftarrow \tensor{Z}_{n-1} \times_{2,4}^{1,3} \cdots \times_{2,4}^{1,3} \tensor{Z}_{1} \times_{2,4}^{1,3} \tensor{Z}_{N} \times_{2,4}^{1,3} \cdots \times_{2,4}^{1,3} \tensor{Z}_{n+1}$
		\State $\tensor{G}^{\ne n} \leftarrow \tensor{G}_{n+1} \boxtimes_2 \cdots \boxtimes_2 \tensor{G}_{N} \boxtimes_2 \tensor{G}_{1} \boxtimes_2 \cdots \boxtimes_2 \tensor{G}_{n-1}$ \Comment From \Cref{def:subchain_tensor,def:subchain_product}
		\State $\mat{M}_{n} \leftarrow \mat{X}_{[n]} \mat{G}_{[2]}^{\ne n}$ 
		\State Solve $\mat{G}_{n(2)} \mat{H}^{\ne n}_{<2>} = \mat{M}_{n}$ 
		\State Recompute $\tensor{Z}_{n} = \sum_{i_n = 1}^{I_n} \mat{G}_{n}(i_n)^\intercal \circ \mat{G}_{n}(i_n)^\intercal$ for the updated TR-core $\tensor{G}_{n}$
		\EndFor
		\Until termination criteria met
	\end{algorithmic}
\end{algorithm}

As mentioned in \Cref{sec:introduction}, randomized methods have been proposed for TR-ALS \cite{malik2021SamplingBasedMethod,malik2022MoreEfficient,yu2022PracticalSketchingBased}. Among them, the most relevant algorithms to this paper are TR-ALS-Sampled \cite{malik2021SamplingBasedMethod} and TR-KSRFT-ALS \cite{yu2022PracticalSketchingBased}. The sampling techniques of these two algorithms will be detailed in \Cref{alg:sstp-st} after introducing an additional definition.

\begin{definition} [Slices-Hadamard product \cite{yu2022PracticalSketchingBased}]
\label{def:hada_product}
    The \textbf{mode-2 slices-Hadamard product} of two tensors $\tensor{A}$ and $\tensor{B}$ is a tensor of size $I_1 \times J \times I_2$ denoted by $\tensor{A} \boxast_2 \tensor{B}$ and defined as
	\begin{equation*}
	(\tensor{A} \boxast_2 \tensor{B})(j) = \mat{A}(j)\mat{B}(j).
	\end{equation*}
\end{definition}

\begin{algorithm}
\caption{Sampled subchain and input tensors (SSIT), summarized from \cite{malik2021SamplingBasedMethod} and \cite{yu2022PracticalSketchingBased}}
\label{alg:sstp-st}
\textbf{Input:} TR-cores $\{\tensor{G}_k\in\bb{R}^{R_k \times I_k \times R_{k+1}}\}_{k=1,k\ne n}^N$, sampling size $m$, probability distributions $\{\vect{p}_k\}_{k=1,k\ne n}^N$

\textbf{Output:} sampled subchain tensor $\tensor{G}^{\ne n}_S$, sampled input tensor $\mat{X}_{S[n]}$

\begin{algorithmic}[1]
    \State $\texttt{idxs} \leftarrow \textsc{Zeros}(m, N-1)$
	\For{$k = n+1, \cdots, N, 1, \cdots, n-1$}
		\State $\texttt{idxs}(:,k) \leftarrow \textsc{Randsample}(I_k, m, true, \vect{p}_k)$
	\EndFor
	\State Let $\tensor{G}^{\ne n}_S$ be a tensor of size $R_{n+1} \times m \times R_{n+1}$, where every lateral slice is an $R_{n+1} \times R_{n+1}$ identity matrix
	\For{$k = n+1, \cdots, N, 1, \cdots, n-1$}
		\State $\tensor{G}^{\ne n}_S \leftarrow \tensor{G}^{\ne n}_S \boxast_2 \tensor{G}_{k}(:, \texttt{idxs}(:,k), :)$
	\EndFor
	\State $\mat{X}_{S[n]} \leftarrow \textsc{Mode-n-Unfolding}(\tensor{X}(\texttt{idxs}(:,1), \cdots, \texttt{idxs}(:,n-1),:,\texttt{idxs}(:,n+1),\cdots \texttt{idxs}(:,N)))$		
\end{algorithmic}
\end{algorithm}

\section{Proposed Methods}
\label{sec:proposed}
We first propose a streaming algorithm for tracking TR decomposition, and then present its randomized variant.
After that, three different sketching techniques based on uniform sampling, leverage-based sampling, and KSRFT, are  discussed.

\subsection{Streaming TR Decomposition}
\label{ssec:str}
Let $\tensor{X}^{old} \in \bb{R}^{I_1 \times \cdots \times I_{N-1} \times t^{old}}$ 
with the $N$-th mode being the time, and its TR decomposition be $\TR \left( \{\tensor{G}_n^{old}\}_{n=1}^N \right)$. Now assume that, at the \emph{time step} $\tau$, a \emph{temporal slice} $\tensor{X}^{new} \in \bb{R}^{I_1 \times \cdots \times I_{N-1} \times t^{new}}$ is added to $\tensor{X}^{old}$ to form a tensor $\tensor{X} \in \bb{R}^{I_1 \times \cdots \times I_{N-1} \times (t^{old} + t^{new}) }$, where $t^{old} \gg t^{new}$. We are interested in finding the TR decomposition $\TR \left( \{\tensor{G}_n\}_{n=1}^N \right)$ of $\tensor{X}$ with the help of $\TR \left( \{\tensor{G}_n^{old}\}_{n=1}^N \right)$ and the existing intermediate information.
In the following, we give the detailed updating formulations.

\paragraph{Update Temporal Mode}
We first consider the update for the TR-core of the temporal mode, i.e., $\tensor{G}_{N}$, by fixing the other TR-cores.
Specifically, by \eqref{eq:tr_als}, we have
\begin{align*}
\mat{G}_{N(2)} 
&\leftarrow \mathop{\arg\min}_{\mat{G}_{N(2)}} \frac{1}{2} \left\|\mat{X}_{[N]} - \mat{G}_{N(2)} (\mat{G}_{[2]}^{\ne N} )^\intercal \right\|_F \\
&= \mathop{\arg\min}_{\mat{G}_{N(2)}} \frac{1}{2} \left\| \begin{bmatrix} \mat{X}_{[N]}^{old} \\ \mat{X}_{[N]}^{new} \end{bmatrix} -  \begin{bmatrix} \mat{G}_{N(2)}^{(1)} \\ \mat{G}_{N(2)}^{(2)} \end{bmatrix} (\mat{G}_{[2]}^{\ne N})_{[2]}^\intercal \right\|_F.
\end{align*}
With \Cref{prop:subchain_gram}  and the fact from \cite{yu2022PracticalSketchingBased},
\begin{equation}
\label{eq:subchain_new}
	\tensor{G}^{\ne n} = \tensor{G}_{n+1} \boxtimes_2 \cdots \boxtimes_2 \tensor{G}_{N} \boxtimes_2 \tensor{G}_{1} \boxtimes_2 \cdots \boxtimes_2 \tensor{G}_{n-1},
\end{equation} 
it is clear that 
\begin{equation*}
\mat{G}_{N(2)} 
\leftarrow \begin{bmatrix} \mat{X}_{[N]}^{old} \mat{G}_{[2]}^{\ne N} \left( (\mat{G}_{[2]}^{\ne N})^\intercal \mat{G}_{[2]}^{\ne N} \right)^\dagger \\ \mat{X}_{[N]}^{new} \mat{G}_{[2]}^{\ne N} \left( (\mat{G}_{[2]}^{\ne N})^\intercal \mat{G}_{[2]}^{\ne N} \right)^\dagger \end{bmatrix} 
= \begin{bmatrix} \mat{G}_{N(2)}^{old} \\ \mat{X}_{[N]}^{new} \mat{G}_{[2]}^{\ne N} \left( \mat{H}^{\ne N}_{<2>} \right)^\dagger \end{bmatrix} 
= \begin{bmatrix} \mat{G}_{N(2)}^{old} \\ \mat{G}_{N(2)}^{new} \end{bmatrix},
\end{equation*}
where $\tensor{H}^{\ne N} = \tensor{Z}_{N-1} \times_{2,4}^{1,3} \cdots \times_{2,4}^{1,3} \tensor{Z}_{1}$ with $\tensor{Z}_{j} = \sum_{i_j=1}^{I_j} \mat{G}_{j}(i_j)^\intercal \circ \mat{G}_{j}(i_j)^\intercal$. 
Thus,
\begin{align*}
	&\mat{G}_{N(2)}^{new} \leftarrow \mat{X}_{[N]}^{new} \mat{G}_{[2]}^{\ne N} \left( \mat{H}^{\ne N}_{<2>} \right)^\dagger, \quad
	\mat{G}_{N(2)} \leftarrow \begin{bmatrix} \mat{G}_{N(2)}^{old} \\ \mat{G}_{N(2)}^{new} \end{bmatrix}.
\end{align*}

\paragraph{Update Non-temporal Modes}
For each non-temporal mode $n \in [N-1]$, we now consider the update of $\tensor{G}_{n} $ by fixing the remain TR-cores. Specifically,
according to \eqref{eq:tr_als}, we have the following normal equation
\begin{align}
\label{eq:partial}
0&=\underbrace{ \mat{X}_{[n]} \left( \underset{\substack{n+1,\cdots,N,\\1,\cdots,n-1}} {\boxtimes_2} \tensor{G}_{j} \right)_{[2]} }_{\mat{P}_{n}} - \mat{G}_{n(2)} \underbrace{ \left( \underset{\substack{n+1,\cdots,N,\\1,\cdots,n-1}} {\boxtimes_2} \tensor{G}_{j} \right)_{[2]}^\intercal \left( \underset{\substack{n+1,\cdots,N,\\1,\cdots,n-1}} {\boxtimes_2} \tensor{G}_{j} \right)_{[2]} }_{\mat{Q}_{n}}\nonumber\\
&=\begin{bmatrix} \mat{X}_{[n]}^{old} & \mat{X}_{[n]}^{new} \end{bmatrix} \begin{bmatrix} \left(\mat{G}^{\ne n}_{old} \right)_{[2]} \\ \left(\mat{G}^{\ne n}_{new} \right)_{[2]} \end{bmatrix} - \mat{G}_{n(2)} \begin{bmatrix} \left(\mat{G}^{\ne n}_{old} \right)_{[2]}^\intercal & \left(\mat{G}^{\ne n}_{new} \right)_{[2]}^\intercal \end{bmatrix}  \begin{bmatrix} \left(\mat{G}^{\ne n}_{old} \right)_{[2]} \\ \left(\mat{G}^{\ne n}_{new} \right)_{[2]} \end{bmatrix}\nonumber \\
&=\left(\mat{P}_{n}^{old} + \mat{X}_{[n]}^{new}  \left(\mat{G}^{\ne n}_{new} \right)_{[2]}\right) - \mat{G}_{n(2)} \left( \mat{Q}_{n}^{old} + \left( \mat{H}^{\ne n}_{new} \right)_{<2>} \right),
\end{align}
where 
\begin{align}
\tensor{G}^{\ne n}_{old} &= \left( \underset{n+1,\cdots,N-1} {\boxtimes_2} \tensor{G}_{j} \right) \boxtimes_2 \tensor{G}_{N}^{old} \boxtimes_2 \left( \underset{1,\cdots,n-1} {\boxtimes_2} \tensor{G}_{j} \right),\label{eq:partial1}\\
\tensor{G}^{\ne n}_{new} &= \left( \underset{n+1,\cdots,N-1} {\boxtimes_2} \tensor{G}_{j} \right) \boxtimes_2 \tensor{G}_{N}^{new} \boxtimes_2 \left( \underset{1,\cdots,n-1} {\boxtimes_2} \tensor{G}_{j} \right),\label{eq:partial2}\\
\tensor{H}^{\ne n}_{new} &= \left( \underset{n-1,\cdots,1} {\times_{2,4}^{1,3}} \tensor{Z}_{j} \right) \times_{2,4}^{1,3} \tensor{Z}_{N}^{new} \times_{2,4}^{1,3} \left( \underset{N-1,\cdots,n+1} {\times_{2,4}^{1,3}} \tensor{Z}_{j} \right)\nonumber
\end{align}
with
\begin{align*}
\tensor{Z}_{N}^{new} =\sum_{i_N=1}^{I_N} \mat{G}^{new}_{N}(i_N)^\intercal \circ \mat{G}^{new}_{N}(i_N)^\intercal.
\end{align*} 
Note that to derive \eqref{eq:partial}, \Cref{prop:subchain_gram}, \eqref{eq:subchain_new}, and the permutation matrix $\mat{\Pi}_{n}$ defined as 
\begin{equation*}
    \left(\mat{X}_{[n]} \mat{\Pi}_{n}\right)(:,\overline{i_1 \cdots i_{n-1} i_{n+1} \cdots i_N}) = \mat{X}_{[n]}(:,\overline{i_{n+1} \cdots i_N i_1 \cdots i_{n-1}})
\end{equation*}
such that 
{\small\begin{align*}
\mat{X}_{[n]} \mat{\Pi}_{n} \mat{\Pi}_{n}^\intercal\left( \underset{\substack{n+1,\cdots,N,\\1,\cdots,n-1}} {\boxtimes_2} \tensor{G}_{j} \right)_{[2]}&=\begin{bmatrix} \mat{X}_{[n]}^{old} & \mat{X}_{[n]}^{new} \end{bmatrix} \begin{bmatrix} \left(\mat{G}^{\ne n}_{old} \right)_{[2]} \\ \left(\mat{G}^{\ne n}_{new} \right)_{[2]} \end{bmatrix}\\
\left( \underset{\substack{n+1,\cdots,N,\\1,\cdots,n-1}} {\boxtimes_2} \tensor{G}_{j} \right)_{[2]}^\intercal \mat{\Pi}_{n} \mat{\Pi}_{n}^\intercal\left( \underset{\substack{n+1,\cdots,N,\\1,\cdots,n-1}} {\boxtimes_2} \tensor{G}_{j} \right)_{[2]}
&=\begin{bmatrix} \left(\mat{G}^{\ne n}_{old} \right)_{[2]}^\intercal & \left(\mat{G}^{\ne n}_{new} \right)_{[2]}^\intercal \end{bmatrix}  \begin{bmatrix} \left(\mat{G}^{\ne n}_{old} \right)_{[2]} \\ \left(\mat{G}^{\ne n}_{new} \right)_{[2]} \end{bmatrix}
\end{align*}}
have been used.
Thus, we achieve the update for $\tensor{G}_{n} $ as follows
\begin{align*}
 \mat{P}_{n} \leftarrow \mat{P}_{n}^{old} +  \mat{X}_{[n]}^{new}  &\left(\mat{G}^{\ne n}_{new} \right)_{[2]}, \quad
 \mat{Q}_{n} \leftarrow \mat{Q}_{n}^{old} + \left( \mat{H}^{\ne n}_{new} \right)_{<2>} ,\\
& \mat{G}_{n(2)} \leftarrow \mat{P}_{n} \mat{Q}_{n}^{\dagger}.
\end{align*}

The whole process for streaming TR decomposition is summarized in \Cref{alg:str}, from which we find that the information of previous decomposition can be stored in the complementary matrices $\mat{P}_{n}$ and $\mat{Q}_{n}$, and hence the expensive computation can be avoided and the TR-cores can be efficiently updated in an incremental way.

\begin{algorithm}
\caption{Streaming TR decomposition (STR)}
\label{alg:str}
\textbf{Input:} Initial tensor $\tensor{X}^{init}$, TR-ranks $R_1, \cdots, R_N$ and new data tensor $\tensor{X}^{new}$

\textbf{Output:} TR-cores $\{ \tensor{G}_n \in \bb{R}^{R_n \times I_n \times R_{n+1}} \}_{n=1}^N$
\begin{algorithmic}[1]
	\Statex // \texttt{Initialization stage}
	\State Compute initial TR-cores $\tensor{G}_{1}, \cdots, \tensor{G}_{N}$ of $\tensor{X}^{init}$ \label{line:str_initial}
	\State Compute the Gram tensors $\tensor{Z}_{1} = \sum_{i_1 = 1}^{I_1} \mat{G}_{1}(i_1)^\intercal \circ \mat{G}_{1}(i_1)^\intercal, \cdots, \tensor{Z}_{N-1} = \sum_{i_{N-1} = 1}^{I_{N-1}} \mat{G}_{N-1}(i_{N-1})^\intercal \circ \mat{G}_{N-1}(i_{N-1})^\intercal$ 
	\For{$n = 1, \cdots, N-1$}
	\State $\mat{P}_{n} \leftarrow \mat{X}_{[n]}^{init} \mat{G}^{\ne n}_{[2]}$
	\State $\mat{Q}_{n} \leftarrow \left( \underset{\substack{n-1,\cdots,1,\\N,\cdots,n+1}} {\times_{2,4}^{1,3}} \tensor{Z}_{j} \right)_{<2>}$
	\EndFor
	\For{$\tau = 1, \cdots, t$ time steps}
	\Statex // \texttt{Update stage for temporal mode}
        \State $\tensor{H}^{\ne N} \leftarrow \tensor{Z}_{N-1} \times_{2,4}^{1,3} \cdots \times_{2,4}^{1,3} \tensor{Z}_{1}$
	\State $\mat{G}_{N(2)}^{new} \leftarrow \mat{X}_{[N]}^{new} \mat{G}_{[2]}^{\ne N} \left( \mat{H}^{\ne N}_{<2>} \right)^\dagger$
	\State $\mat{G}_{N(2)} \leftarrow \begin{bmatrix} \mat{G}_{N(2)}^{old} \\ \mat{G}_{N(2)}^{new} \end{bmatrix}$ and reshape $\mat{G}_{N(2)}$ to $\tensor{G}_{N}$
	\State $\tensor{Z}_{N} \leftarrow \sum_{i_{N} = 1}^{I_{N}} \mat{G}_{N}(i_{N})^\intercal \circ \mat{G}_{N}(i_{N})^\intercal$
	\Statex // \texttt{Update stage for non-temporal modes}
	\For{$n = 1, \cdots, N-1$}
        \State $\tensor{G}^{\ne n}_{new} \leftarrow \left( \underset{n+1,\cdots,N-1} {\boxtimes_2} \tensor{G}_{j} \right) \boxtimes_2 \tensor{G}_{N}^{new} \boxtimes_2 \left( \underset{1,\cdots,n-1} {\boxtimes_2} \tensor{G}_{j} \right)$
	\State $\mat{P}_{n} \leftarrow \mat{P}_{n} +  \mat{X}_{[n]}^{new}  \left(\mat{G}^{\ne n}_{new} \right)_{[2]}$
	\State $\tensor{H}^{\ne n}_{new} = \left( \underset{n-1,\cdots,1} {\times_{2,4}^{1,3}} \tensor{Z}_{j} \right) \times_{2,4}^{1,3} \tensor{Z}_{N}^{new} \times_{2,4}^{1,3} \left( \underset{N-1,\cdots,n+1} {\times_{2,4}^{1,3}} \tensor{Z}_{j} \right)$
	\State $\mat{Q}_{n} \leftarrow \mat{Q}_{n} + \left( \mat{H}^{\ne n}_{new} \right)_{<2>}$
	\State $\mat{G}_{n(2)} \leftarrow \mat{P}_{n} \mat{Q}_{n}^{\dagger}$ and reshape $\mat{G}_{n(2)}$ to $\tensor{G}_{n}$
	\State $\tensor{Z}_{n} \leftarrow \sum_{i_{n} = 1}^{I_{n}} \mat{G}_{n}(i_{n})^\intercal \circ \mat{G}_{n}(i_{n})^\intercal$
	\EndFor
	\EndFor
\end{algorithmic}
\end{algorithm}

\begin{remark}
    With regard to the initialization of streaming TR decomposition, i.e., \Cref{line:str_initial} in \Cref{alg:str}, 
    we can choose any feasible techniques. 
    Inspired by 
    the experimental results in \cite[Section III.B]{ma2018RandomizedOnline}, we recommend running the corresponding offline version of \Cref{alg:str} for finding the initial values. 
    However, in the specific numerical experiments later in this paper, we use the same initial values for various algorithms for convenience; see the detailed description of  experiments in \Cref{sec:experiments}. The above explanation is also applicable to \Cref{alg:rstr} below.
\end{remark}

\begin{remark}
From the derivation of \Cref{alg:str}, it can be seen that the structure of the coefficient matrices in the subproblems is well used. That is, \Cref{prop:subchain_gram} is employed to reduce the computational cost. 
Hence, the algorithm is more efficient than applying TR-ALS directly. 
More descriptions and comparisons on advantages for using \Cref{prop:subchain_gram} can be found in \cite{yu2022PracticalAlternating}.
\end{remark}

\begin{remark}
As we know, TR decomposition generalizes the famous TT decomposition by relaxing some constraints  \cite{zhao2016TensorRing}. So, with a slight change, \Cref{alg:str} is also applicable to TT decomposition. It is worthy to emphasize that the corresponding method is very different from the ones in \cite{liu2021IncrementalTensorTrain,thanh2021AdaptiveAlgorithms,kressner2022StreamingTensor} mentioned in \Cref{sec:introduction}.  
The main difference still lies in that we make full use of the structure introduced before. 
\end{remark}

\subsection{Randomized Streaming TR Decomposition}
\label{ssec:rstr}
We now employ randomized sketching techniques to improve the efficiency of streaming TR decomposition. That is, we consider the following sketched subproblems for streaming tensor with the sketching matrices $\mat{\Psi}_{n} \in \bb{R}^{m \times \prod_{j \ne N}I_j}$,
\begin{equation}
\label{eq:rtr_als}
\mathop{\arg\min}_{\mat{G}_{n(2)}} \left\|\mat{\Psi}_n \mat{G}_{[2]}^{\ne n} \mat{G}_{n(2)}^\intercal-\mat{\Psi}_n \mat{X}_{[n]}^\intercal \right\|_F,\ n=1,\cdots, N.
\end{equation}
A randomized streaming TR decomposition will be proposed.  
In the following, we give the specific updating rules. 

\paragraph{Update Temporal Mode}
By dividing the corresponding terms into two parts, from \eqref{eq:rtr_als}, we have 
\begin{align*}
\mat{G}_{N(2)} 
&\leftarrow 
\mathop{\arg\min}_{\mat{G}_{N(2)}} \frac{1}{2} \left\| \begin{bmatrix} \mat{X}_{[N]}^{old} (\mat{\Psi}_{N}^{old})^\intercal \\ \mat{X}_{[N]}^{new} (\mat{\Psi}_{N}^{new})^\intercal \end{bmatrix}  -  \begin{bmatrix} \mat{G}_{N(2)}^{(1)} \left( \mat{\Psi}_{N}^{old} \mat{G}_{[2]}^{\ne N} \right)^\intercal \\ \mat{G}_{N(2)}^{(2)} \left( \mat{\Psi}_{N}^{new} \mat{G}_{[2]}^{\ne N} \right)^\intercal \end{bmatrix} \right\|_F,
\end{align*}
where $\mat{\Psi}_{N}^{old} \in \bb{R}^{m \times \prod_{j \ne N}I_j}$ and $\mat{\Psi}_{N}^{new} \in \bb{R}^{m \times \prod_{j \ne N}I_j}$.
It is clear that 
\begin{align*}
\mat{G}_{N(2)} 
&\leftarrow \begin{bmatrix} \mat{X}_{[N]}^{old} (\mat{\Psi}_{N}^{old})^\intercal \left( \left( \mat{\Psi}_{N}^{old} \mat{G}_{[2]}^{\ne N} \right)^\intercal \right)^\dagger \\ \mat{X}_{[N]}^{new} (\mat{\Psi}_{N}^{new})^\intercal \left( \left( \mat{\Psi}_{N}^{new} \mat{G}_{[2]}^{\ne N} \right)^\intercal \right)^\dagger  \end{bmatrix} \\
&= \begin{bmatrix} \mat{G}_{N(2)}^{old} \\ \mat{X}_{[N]}^{new} (\mat{\Psi}_{N}^{new})^\intercal \left( \left( \mat{\Psi}_{N}^{new} \mat{G}_{[2]}^{\ne N} \right)^\intercal \right)^\dagger  \end{bmatrix} 
= \begin{bmatrix} \mat{G}_{N(2)}^{old} \\ \mat{G}_{N(2)}^{new} \end{bmatrix}.
\end{align*}
Thus,
\begin{align*}
	&\mat{G}_{N(2)}^{new} \leftarrow \mat{X}_{[N]}^{new} (\mat{\Psi}_{N}^{new})^\intercal \left( \left( \mat{\Psi}_{N}^{new} \mat{G}_{[2]}^{\ne N} \right)^\intercal \right)^\dagger, \quad
	\mat{G}_{N(2)} \leftarrow \begin{bmatrix} \mat{G}_{N(2)}^{old} \\ \mat{G}_{N(2)}^{new} \end{bmatrix}.
\end{align*}

\begin{remark}
    Usually, the sketching matrix $\mat{\Psi}_{N}^{new}$ is not the same as $\mat{\Psi}_{N}^{old}$, which implies that $\mat{X}_{[N]}^{new}$ will be 
    sketched at each time step.
\end{remark}

\paragraph{Update Non-temporal Modes}
As done for streaming TR decomposition in \Cref{ssec:str} and similar to the above deduction, we have
\begin{align*}
& \mat{P}_{n} \leftarrow \mat{P}_{n}^{old} +  \mat{X}_{[n]}^{new}  (\mat{\Psi}_{n}^{new})^\intercal \mat{\Psi}_{n}^{new} \left(\mat{G}^{\ne n}_{new} \right)_{[2]}, \\
& \mat{Q}_{n} \leftarrow \mat{Q}_{n}^{old} + \left(\mat{G}^{\ne n}_{new} \right)_{[2]}^\intercal (\mat{\Psi}_{n}^{new})^\intercal \mat{\Psi}_{n}^{new} \left(\mat{G}^{\ne n}_{new} \right)_{[2]}, \\
& \mat{G}_{n(2)} \leftarrow \mat{P}_{n} \mat{Q}_{n}^{\dagger},
\end{align*}
where $\mat{\Psi}_{n}^{old} \in \bb{R}^{m \times \prod_{j \ne n}I_j}$ and $\mat{\Psi}_{n}^{new} \in \bb{R}^{m \times \prod_{j \ne n}I_j}$ are sketching matrices and the other notations are the same as the ones in \eqref{eq:partial}, \eqref{eq:partial1}, and \eqref{eq:partial2}.

\begin{remark}
Unlike the case for streaming TR decomposition in \Cref{ssec:str}, 
here the calculation of the Gram matrix $\left(\mat{G}^{\ne n}_{new} \right)_{[2]}^\intercal (\mat{\Psi}_{n}^{new})^\intercal \mat{\Psi}_{n}^{new} \left(\mat{G}^{\ne n}_{new} \right)_{[2]}$ is quite cheap. So, we do not consider its structure any more though it still exists. Instead, we mainly focus on how to compute  
$\mat{\Psi}_{n}^{new} \left(\mat{G}^{\ne n}_{new} \right)_{[2]}$ fast by using the structure of $\left(\mat{G}^{\ne n}_{new} \right)_{[2]}$ and choosing suitable 
$\mat{\Psi}_{n}^{new}$; see 
\Cref{ssec:sketching} below.	
\end{remark}

The whole process for randomized streaming TR decomposition is summarized in \Cref{alg:rstr}, which shows that, as carried out by \Cref{alg:sstp-st} or \Cref{alg:sketchST}, different sketching techniques can be used to compute the sketched subchain and input tensors. Moreover, when forming the aforementioned sketched tensors, the un-updated TR-cores and fibers do not need to be sketched again. The corresponding detailed algorithms are presented in \Cref{SUPP:alg}. Note that, in this case, the theoretical guarantees given in \Cref{ssec:sketching} still apply.

\begin{algorithm}
\caption{Randomized streaming TR decomposition (rSTR)}
\label{alg:rstr}
\textbf{Input:} Initial tensor $\tensor{X}^{init}$, TR-ranks $R_1, \cdots, R_N$, new data tensor $\tensor{X}^{new}$ and sketch size $m$

\textbf{Output:} TR-cores $\{ \tensor{G}_n \in \bb{R}^{R_n \times I_n \times R_{n+1}} \}_{n=1}^N$
\begin{algorithmic}[1]
	\Statex // \texttt{Initialization stage}
	\State Compute initial TR-cores $\tensor{G}_{1}, \cdots, \tensor{G}_{N}$ of $\tensor{X}^{init}$ \label{line:rand-init}
	\For{$n = 1, \cdots, N-1$}
	\State Compute $\tensor{G}_{S}^{\ne n}$ and $\mat{X}_{S[n]}^{init}$ using \Cref{alg:sstp-st} or \Cref{alg:sketchST}
	\State $\mat{P}_{n} \leftarrow \mat{X}_{S[n]}^{init} \mat{G}^{\ne n}_{S[2]}$
	\State $\mat{Q}_{n} \leftarrow \left( \mat{G}^{\ne n}_{S[2]} \right)^\intercal \mat{G}^{\ne n}_{S[2]}$
	\EndFor
	\For{$\tau = 1, \cdots, t$ time steps}
	\Statex // \texttt{Update stage for temporal mode}
	\State Compute $\tensor{G}_{S}^{\ne N}$ and $\mat{X}_{S[N]}^{new}$ using \Cref{alg:sstp-st} or \Cref{alg:sketchST}
	\State $\mat{G}_{N(2)}^{new} \leftarrow \mat{X}_{S[N]}^{new} \left( \left( \mat{G}_{S[2]}^{\ne N} \right)^\intercal \right)^\dagger$
	\State $\mat{G}_{N(2)} \leftarrow \begin{bmatrix} \mat{G}_{N(2)}^{old} \\ \mat{G}_{N(2)}^{new} \end{bmatrix}$ and reshape $\mat{G}_{N(2)}$ to $\tensor{G}_{N}$
	\Statex // \texttt{Update stage for non-temporal modes}
	\For{$n = 1, \cdots, N-1$}
        \State $\tensor{G}^{\ne n}_{new} \leftarrow \left( \underset{n+1,\cdots,N-1} {\boxtimes_2} \tensor{G}_{j} \right) \boxtimes_2 \tensor{G}_{N}^{new} \boxtimes_2 \left( \underset{1,\cdots,n-1} {\boxtimes_2} \tensor{G}_{j} \right)$
	\State Compute $\left(\tensor{G}_{new}^{\ne n}\right)_{S}$ and $\mat{X}_{S[n]}^{new}$ using \Cref{alg:sstp-st} or \Cref{alg:sketchST}
	\State $\mat{P}_{n} \leftarrow \mat{P}_{n} +  \mat{X}_{S[n]}^{new}  \left(\mat{G}^{\ne n}_{new} \right)_{S[2]}$
	\State $\mat{Q}_{n} \leftarrow \mat{Q}_{n} + \left(\mat{G}^{\ne n}_{new} \right)_{S[2]}^\intercal \left(\mat{G}^{\ne n}_{new} \right)_{S[2]}$
	\State $\mat{G}_{n(2)} \leftarrow \mat{P}_{n} \mat{Q}_{n}^{\dagger}$ and reshape $\mat{G}_{n(2)}$ to $\tensor{G}_{n}$
	\EndFor
	\EndFor
\end{algorithmic}
\end{algorithm}


\subsection{Different Sketching Techniques}
\label{ssec:sketching}
We mainly consider three practical sketching techniques: Uniform sampling, leverage-based sampling, and KSRFT.

\paragraph{Uniform Sampling}
That is, 
\begin{equation}
\label{eq:sampDS}
	\mat{\Psi}_{n} = \mat{D}_{n} \mat{S}_{n},
\end{equation}
where $\mat{S}_{n} \in \bb{R}^{m \times J_n}$ with
$$J_n=
\begin{cases}
	I_{1} I_{2} \cdots I_{N-1}, & n=N \\
	I_{1} I_{2} \cdots I_{n-1} I_{n+1} \cdots I_{N-1} t^{new}, & n \ne N
\end{cases}
$$
is a sampling matrix, 
i.e.,
\begin{equation*}
(\mat{S}_{n})_{ij}=
\begin{cases}
	1, & \text{if the $j$-th row is chosen in the $i$-th independent random trial} \\
     & \text{with the probability $1/J_n$,}\\
	0, & \text{otherwise,}
\end{cases}
\end{equation*}
and $\mat{D}_{n} \in \bb{R}^{m \times m}$ is a diagonal rescaling matrix with the $i$-th diagonal entry $(\mat{D}_{n})_{ii} = \sqrt{J_n/m}$. In practice, the rescaling matrix  can be ignored without affecting the performance of algorithms. 
Furthermore, the sampling can be carried out in TR-cores as done in \Cref{alg:sstp-st}.
The detailed algorithm is summarized in \Cref{alg:rstr_uni} in \Cref{SUPP:alg} and the theoretical guarantee is as follows.

\begin{theorem}
\label{thm:uni}
Let $\mat{\Psi}_n$ be a uniform sampling matrix as defined above and 
$$\tilde{\mat{G}}_{n(2)} \defeq \arg\min_{\mat{G}_{n(2)}} \left\| \mat{X}_{[n]} \mat{\Psi}_n^\intercal - \mat{G}_{n(2)}(\mat{\Psi}_n\mat{G}_{[2]}^{\ne n})^\intercal \right\|_F.$$ 
If 
\begin{equation*}
	m \geq \left( \frac{2 \gamma R_{n} R_{n+1}}{\varepsilon} \right) \max \left[ \frac{48}{\varepsilon}\ln \left( \frac{96 \gamma R_{n} R_{n+1}}{ \varepsilon^2 \sqrt{\delta}} \right), \frac{1}{\delta} \right]  
\end{equation*}
with $\varepsilon \in (0,1)$, $\delta \in (0,1)$, and $\gamma>1$, then the following inequality holds with a probability of at least $1-\delta$:
\begin{equation*}
	\left\| \mat{X}_{[n]}  - \tilde{\mat{G}}_{n(2)}(\mat{G}_{[2]}^{\ne n})^\intercal \right\|_F \leq (1+\varepsilon) \min_{\mat{G}_{n(2)}} \left\| \mat{X}_{[n]} - \mat{G}_{n(2)}(\mat{G}_{[2]}^{\ne n})^\intercal \right\|_F.
\end{equation*}
\end{theorem}

\begin{remark}
The $\gamma$ in \Cref{thm:uni} determines the maximum value of the row norms of the left singular vector matrix of the coefficient matrix $\mat{G}_{[2]}^{\ne n}$ (see \Cref{SUPP:thm:uni-offline}). 
It can be seen that the more inhomogeneous the coefficient matrix is, the larger the $\gamma$ is, which leads to less effective for uniform sampling. In this case, the importance sampling is a more reasonable choice. 
\end{remark}

\paragraph{Leverage-based Sampling}
Two definitions are first introduced.
\begin{definition}[Leverage Scores \cite{drineas2012FastApproximation}]
\label{def:leverages_scores}
	Let $\mat{A}\in \bb{R}^{m\times n}$ with $m > n$, and
	let $\mat{Q} \in \bb{R}^{m \times n}$ be any orthogonal basis for the column space of $\mat{A}$.
	The \textbf{leverage score} of the $i$-th row of $\mat{A}$ is given by
	\begin{equation*}
		\ell_i(\mat{A}) = \|\mat{Q}(i,:)\|_2^2.
	\end{equation*}
\end{definition}

\begin{definition}[Leverage-based Probability Distribution \cite{woodruff2014SketchingTool}]
\label{def:leverages_scores_sampling}
	Let $\mat{A}\in \bb{R}^{m\times n}$ with $m > n$.
	We say  a probability distribution $\vect{q}=[q_1, \cdots, q_{m} ]^\intercal$ is a \textbf{leverage-based probability distribution} for $\mat{A}$ on $[m]$ if $q_i \ge \beta p_i$ with $p_i=\frac{\ell_i(\mat{A})}{n}$, $0<\beta \le 1$ and $\forall i \in [m]$.
\end{definition}

Computing the leverage scores of $\mat{G}^{\ne n}_{[2]} \in \bb{R}^{J_n \times R_{n}R_{n+1}}$ directly is expensive. 
Fortunately, by \cite{malik2021SamplingBasedMethod}, we can estimate them using the leverage scores related to the TR-cores $\tensor{G}_{1}, \cdots, \tensor{G}_{n-1},\tensor{G}_{n+1}, \cdots, \tensor{G}_{N}$. 

\begin{lemma}[\cite{malik2021SamplingBasedMethod}]
\label{lem:lev_est}
	For each $n \in [N]$, let $\vect{p}_n \in \bb{R}^{I_n}$ be a probability distribution on $[I_n]$ defined element-wise via
	\begin{equation*}
		\vect{p}_n(i_n) = \frac{\ell_{i_n}(\mat{G}_{n(2)})}{\rank(\mat{G}_{n(2)}))},
	\end{equation*}
	$\vect{p}^{\ne n}$ be a probability distribution on $\left[J_n\right]$ defined element-wise via
	\begin{equation*}
		\vect{p}^{\ne n}(i) = \frac{\ell_{i}(\mat{G}^{\ne n}_{[2]})}{\rank(\mat{G}^{\ne n}_{[2]})},
	\end{equation*}
	$\vect{q}^{\ne n}$ be a vector defined element-wise via
	\begin{equation*}
		\vect{q}^{\ne n}(\overline{i_{n+1} \cdots i_{N} i_{1} \cdots i_{n-1}}) = \prod_{\substack{m=1\\m \ne n}}^{N} \vect{p}_m(i_m),
	\end{equation*} 
	and $\beta_n$ be a constant as in \Cref{def:leverages_scores_sampling} defined as 
	\begin{equation*}
		\beta_{n} = \left( R_{n}R_{n+1} \prod_{\substack{m=1\\m\notin \{n, n+1\}}}^{N} R_{m}^2 \right)^{-1}.
	\end{equation*}
	Then for each $n \in [N]$, $\vect{q}^{\ne n}(i) \geq \beta_n \vect{p}^{\ne n}(i)$ for all $i = \overline{i_{n+1} \cdots i_{N} i_{1} \cdots i_{n-1}} \in \left[J_n\right]$ and hence $\vect{q}^{\ne n}$ is the leverage-based probability distribution for $\mat{G}^{\ne n}_{[2]}$ on $\left[J_n\right]$.
\end{lemma}

With this lemma, we can define the sampling matrix in \eqref{eq:sampDS} as follows:
\begin{equation*}
(\mat{S}_{n})_{ij}=
\begin{cases}
	1, & \text{if the $j$-th row is chosen in the $i$-th independent random trial} \\
     & \text{ with the probability $\vect{q}^{\ne n}(j)$,}\\
	0, & \text{otherwise,}
\end{cases}
\end{equation*}
and the $i$-th diagonal entry of the diagonal rescaling matrix $\mat{D}_{n}$ in \eqref{eq:sampDS} is now $(\mat{D}_{n})_{ii} = 1/\sqrt{m \vect{q}^{\ne n}(j)}$. As above, the rescaling matrix can be ignored and the sampling can be carried out in TR-cores as done in \Cref{alg:sstp-st}. 
The detailed algorithm is summarized in \Cref{alg:rstr_lev} in \Cref{SUPP:alg} and the theoretical guarantee is given in the following.
\begin{theorem}
\label{thm:lev}	
Let $\mat{\Psi}_{n}$ be a leveraged-based sampling matrix as defined above and 
\begin{equation*}
    \tilde{\mat{G}}_{n(2)} \defeq \arg\min_{\mat{G}_{n(2)}} \left\| \mat{X}_{[n]} \mat{\Psi}_{n}^\intercal - \mat{G}_{n(2)}(\mat{\Psi}_{n}\mat{G}_{[2]}^{\ne n})^\intercal \right\|_F.
\end{equation*}
If 
\begin{equation*}
	m > \left( \prod_{j=1}^{N} R_j^2 \right) \max \left[ \frac{16}{3(\sqrt{2}-1)^2} \ln \left( \frac{4R_{n} R_{n+1}}{\delta} \right), \frac{4}{\varepsilon \delta} \right]  
\end{equation*}
with $\varepsilon \in (0,1)$ and $\delta \in (0,1)$, then the following inequality holds with a probability of at least $1-\delta$:
\begin{equation*}
	\left\| \mat{X}_{[n]}  - \tilde{\mat{G}}_{n(2)}(\mat{G}_{[2]}^{\ne n})^\intercal \right\|_F \leq (1+\varepsilon) \min_{\mat{G}_{n(2)}} \left\| \mat{X}_{[n]} - \mat{G}_{n(2)}(\mat{G}_{[2]}^{\ne n})^\intercal \right\|_F.
\end{equation*}
\end{theorem}

\paragraph{KSRFT}
The definition of KSRFT is listed as follows.

\begin{definition}[KSRFT \cite{battaglino2018PracticalRandomized,jin2021FasterJohnson}]\label{def:SRFT}
	The KSRFT is defined as
	\begin{equation*}
	\mat{\Psi} = \sqrt{\frac{\prod_{j=1}^N I_j}{m}}
	\mat{S}  \left(\bigotimes_{j=1}^N (\mat{F}_j \mat{D}_j)\right),
	\end{equation*}
	where
	\begin{itemize}
		\item $\mat{S} \in \bb{R}^{m \times \prod_{j=1}^N I_j}$ : $m$ rows, drawn uniformly with replacement, of the $\prod_{i=j}^N I_j \times \prod_{j=1}^N I_j$ identity matrix, 
        i.e., it is a unform sampling matrix;
		\item $\mat{F}_j \in \bb{C}^{I_j \times I_j}$ : (unitary) discrete Fourier transform (DFT) of dimension $I_j$; 
		\item $\mat{D}_j \in \bb{R}^{I_j \times I_j}$ : a diagonal matrix with independent random diagonal entries drawn uniformly from $\{+1,-1\}$ (also called random sign-flip operator).
	\end{itemize}
\end{definition}

\Cref{alg:sketchST} shows the method for calculating the sketched subchain and input tensors based on KSRFT, which is summarized from \cite{yu2022PracticalSketchingBased}. 
Considering that KSRFT transforms the original TR-ALS subproblems into complex ones, 
the update of TR-cores needs to do the following slight change:
\begin{align*}
&\mat{G}_{N(2)}^{new} \leftarrow \Re\left(\hat{\mat{X}}_{S[N]}^{new} \right)\left( \Re\left( \left( \hat{\mat{G}}_{S[2]}^{\ne N} \right)^\intercal \right)\right)^\dagger,\quad
\mat{G}_{N(2)} \leftarrow \begin{bmatrix} \mat{G}_{N(2)}^{old} \\ \mat{G}_{N(2)}^{new} \end{bmatrix},\\
&\mat{P}_{n} \leftarrow \mat{P}_{n}^{old} + \mat{P}_{n} + \hat{\mat{X}}_{S[n]}^{new}  \overline{\left(\hat{\mat{G}}^{\ne n}_{new} \right)_{S[2]}},\\
&\mat{Q}_{n} \leftarrow \mat{Q}_{n}^{old} + \left(\hat{\mat{G}}^{\ne n}_{new} \right)_{S[2]}^\intercal \left(\hat{\mat{G}}^{\ne n}_{new} \right)_{S[2]},\\
& \mat{G}_{n(2)} \leftarrow \Re(\mat{P}_{n}) \Re(\mat{Q}_{n})^{\dagger},
\end{align*}
where $\hat{\tensor{X}} = \tensor{X} \times_1 (\mat{F}_1 \mat{D}_1) \times_2 (\mat{F}_2 \mat{D}_2) \cdots \times_N (\mat{F}_N \mat{D}_N)$, 
$\hat{\tensor{G}}_{n} = \tensor{G}_{n} \times_2 (\mat{F}_n \mat{D}_n) \textrm{ for }n=1,\cdots,N$, and $\Re(\cdot)$ and $\overline{(\cdot)}$ remain the real-value and conjugation of entries of a matrix, respectively. 
The detailed algorithm is summarized in \Cref{alg:rstr_ksrft} in \Cref{SUPP:alg} and the theoretical guarantee is given in \Cref{thm:ksrft}.

\begin{algorithm}
\caption{Sketched subchain and input tensors based on KSRFT, summarized from \cite{yu2022PracticalSketchingBased}}
\label{alg:sketchST}
\textbf{Input:} TR-cores $\{\tensor{G}_k \in \bb{R}^{R_k \times I_k \times R_{k+1}}\}_{k=1,k\ne n}^N$, sketch size $m$, tensor mode $n$

\textbf{Output:} sketched subchain tensor $\hat{\tensor{G}}^{\ne n}_S$, sketched input tensor $\hat{\mat{X}}_{S[n]}$

\begin{algorithmic}[1]	
	\State Define the random sign-flip operators $\mat{D}_j$ and DFT matrices $\mat{F}_j$ for $j \in [N]$
	\State Mix TR-cores: $\hat{\tensor{G}}_{j} \leftarrow \tensor{G}_{j} \times_2 (\mat{F}_j \mat{D}_j)$, for $j \in [N]\backslash n$ 
	\State Mix tensor: $\hat{\tensor{X}} \leftarrow \tensor{X} \times_1 (\mat{F}_1 \mat{D}_1) \times_2 (\mat{F}_2 \mat{D}_2) \cdots \times_N (\mat{F}_N \mat{D}_N)$ 
	\For{$n = 1, \cdots, N$}
		\State Define sampling operator $\mat{S} \in \bb{R}^{m \times \prod_{j \ne n}I_j}$ 
		\State Retrieve \texttt{idxs} from $\mat{S}$ 
		\State Compute $\hat{\tensor{G}}^{\ne n}_S$ and $\hat{\mat{X}}_{S[n]}$ by \Cref{alg:sstp-st} using uniform sampling
		\State $\hat{\mat{X}}_{S[n]} \leftarrow  \mat{D}_n \mat{F}_n^* \hat{\mat{X}}_{S[n]}$
	\EndFor			
\end{algorithmic}
\end{algorithm}

\begin{theorem}
\label{thm:ksrft}Let $\mat{\Psi}_{n}$ be a KSRFT as defined in \Cref{def:SRFT} and 
$$\tilde{\mat{G}}_{n(2)} \defeq \arg\min_{\mat{G}_{n(2)}} \left\| \mat{X}_{[n]} \mat{\Psi}_{n}^\intercal - \mat{G}_{n(2)}(\mat{\Psi}_{n}\mat{G}_{[2]}^{\ne n})^\intercal \right\|_F.$$ 
If $m \geq \bigO{xy}$
with 
\begin{align*}
x&=\varepsilon^{-1} (R_{n}R_{n+1}) \log^{2N-3}\left(\left(\frac{R_{n}R_{n+1}}{\varepsilon}\right)^{\frac{R_{n}R_{n+1}}{2}} \frac{N-1}{\eta}\right),\\
y&=\log^4\left(\left(\frac{R_{n}R_{n+1}}{\varepsilon}\right)^{\frac{1}{2}} \log^{N-1}\left( \left(\frac{R_{n}R_{n+1}}{\varepsilon}\right)^{\frac{R_{n}R_{n+1}}{2}} \frac{N-1}{\eta} \right) \right) \log \prod_{j \ne n} I_j,
\end{align*}
where $\varepsilon \in (0,1)$ is such that $\prod_{j \ne n} I_j \lesssim 1/\varepsilon^{R_{n}R_{n+1}}$ with $R_{n}R_{n+1} \geq 2$, $\delta \in (0,1)$, and $\eta \in (0,1)$,
then the following inequality holds with a probability of at least $1-\eta-2^{-\Omega(\log \prod_{j \ne n} I_j)}$:
\begin{equation*}
	\left\| \mat{X}_{[n]}  - \tilde{\mat{G}}_{n(2)}(\mat{G}_{[2]}^{\ne n})^\intercal \right\|_F \leq (1+\varepsilon) \min_{\mat{G}_{n(2)}} \left\| \mat{X}_{[n]} - \mat{G}_{n(2)}(\mat{G}_{[2]}^{\ne n})^\intercal \right\|_F.
\end{equation*}
Furthermore, if an assumption on $(N-1)/\eta \leq (R_{n}R_{n+1}/\varepsilon)^{R_{n}R_{n+1}/2}$ also holds, the bound on sketch size can be changed to 
{\scriptsize
\begin{equation*}
	m \geq \bigO{\varepsilon^{-1} (R_{n}R_{n+1})^{2(N-1)} \log^{2N-3}\left(\frac{R_{n}R_{n+1}}{\varepsilon}\right) \log^4\left(\frac{R_{n}R_{n+1}}{\varepsilon} \log\left(\frac{R_{n}R_{n+1}}{\varepsilon}\right)\right) \log \prod_{j \ne n} I_j} .  
\end{equation*}
}
\end{theorem}

\section{Numerical Experiments}
\label{sec:experiments}
 In this section, we consider the numerical performance of our STR and rSTR\footnote{We use rSTR-U, rSTR-L and rSTR-K to notate rSTR with uniform sampling, leverage-based sampling and KSRFT, respectively.}. 
Specifically, we first examine their effectiveness and efficiency on two real-world datasets. 
Then, based on the investigation on synthetic tensors, we show their performance 
from various perspectives in greater detail.
Six baselines have been chosen as competitors 
to evaluate the performance in our experiments:
\begin{itemize}
	\item TR-ALS (Cold) \cite{zhao2016TensorRing}: an implementation of TR-ALS without special initialization.
	\item TR-ALS (Hot): the same 
    as above but the TR decomposition of the last time step is used as the initialization for decomposing the current tensor.
	\item TR-ALS-NE \cite{yu2022PracticalAlternating}: a practical implementation of TR-ALS.
	\item TR-ALS-Sampled-U: a sampling-based algorithm with uniform sampling.
	\item TR-ALS-Sampled \cite{malik2021SamplingBasedMethod}: the same as above but with leverage-based sampling.
	\item TR-KSRFT-ALS \cite{yu2022PracticalSketchingBased}: a practical implementation of KSRFT-based algorithm.
\end{itemize}
The computational complexities of the above methods as well as 
ours are listed in \Cref{tab:complexity}. In addition, our STR and rSTR occupy the memory space of  
$$I^{N-1}t^{new} + (2(N-1)I+t^{old})R^2 + (N-1)R^4,$$
which is much smaller than $I^{N-1}(t^{old}+t^{new})$, the memory space of the other methods.
\begin{table}[htbp]
	\centering
	\caption{Complexity comparison of various algorithms 
   (Ignoring the initialization, i.e., the initial TR-cores in offline algorithms and the initialization stage in STR and rSTR, and setting $\tensor{X} \in \bb{R}^{I \times \cdots \times I \times (t^{old} + t^{new}) }$ with the target TR-ranks $R_1=\cdots=R_N=R$ satisfying  $R^2<I$).} 
	\label{tab:complexity}
	\resizebox{1\linewidth}{!}{
	\begin{tabular}{ll}  
		\toprule
		Method 		        & Time only for one time step \\ 
		\midrule
		TR-ALS 	            & $\bigO{\#it \cdot N I^{N-1} R^2 (t^{old}+t^{new})}$\quad\quad\quad\quad // $\#it$ denotes the number of outer loop iterations\\
		TR-ALS-NE           & $\bigO{(N-1)IR^4 + (t^{old}+t^{new})R^4 + \#it \cdot N I^{N-1} R^2 (t^{old}+t^{new})}$ \\
		TR-ALS-Sampled-U    & $\bigO{\#it \cdot ((N-1)ImR^2 + (t^{old}+t^{new})mR^2)}$ \\
		TR-ALS-Sampled      & $\bigO{(N-1)IR^4 + (t^{old}+t^{new})R^4 + \#it \cdot ((N-1)ImR^2 + (t^{old}+t^{new})mR^2)}$ \\
		TR-KSRFT-ALS        & $\bigO{I^{N-1}(t^{old}+t^{new}) \log(I^{N-1}(t^{old}+t^{new})) + \#it \cdot ((N-1)ImR^2 + (t^{old}+t^{new})mR^2)}$ \\
		\midrule
		STR                 & $\bigO{N I^{N-1} R^2 t^{new}}$ \\
		rSTR-U              & $\bigO{(N-1)ImR^2 + t^{new}mR^2}$ \\
		rSTR-L              & $\bigO{(N-1)ImR^2 + t^{new}mR^2}$ \\
		rSTR-K              & $\bigO{I^{N-1}(t^{new}) \log(I^{N-1}(t^{new}))}$ \\
		\bottomrule
	\end{tabular}
	}
\end{table}	

The experimental protocol is the same for all the experiments.  Specifically, for a given dataset of size $I_{1} \times \cdots \times I_{N}$, 
a subtensor of size $I_{1} \times \cdots \times I_{N-1} \times (20\%I_{N})$ is first decomposed by TR-ALS and the TR decomposition is used to initialize all the algorithms.
After that, a section of size $I_{1} \times \cdots \times I_{N-1} \times t^{new}$ of the remaining 
data is appended 
to the existing tensor at a time step, 
immediately following which 
all the methods record their processing time for this step, as well as calculate the relative errors of their current decompositions by 
\begin{equation*}
	\frac{\left\|\tensor{X} -\hat{\tensor{X}}\right\|_F}{\|  \tensor{X} \|_F}=\frac{\left\| \tensor{X}-\TR \left( \{\hat{\tensor{G}}_n\}_{n=1}^N \right)  \right\|_F}{\| \tensor{X} \|_F},
\end{equation*}
where the TR-cores $\{ \hat{\tensor{G}}_n\}_{n=1}^N$ are computed by various algorithms. Thus, continuing the process, we can report and compare the processing time and relative errors for all the time steps. 
 
The same experiment is replicated 10 times for all datasets by using Matlab R2022a on a computer 
 with an Intel Xeon W-2255 3.7 GHz CPU, and 256 GB RAM, and the final results are averaged over these runs.
 Additionally, we also use the Matlab Tensor Toolbox \cite{kolda2006TensorToolbox}.
 
 For the initialization stage, there are some settings of parameters that need to be clarified. 
 Firstly, since we only care about the  comparison on relative performance among different algorithms, it is not necessary to pursue the best rank decomposition for each dataset. 
 Hence, unless otherwise stated, the target rank $R$ is always fixed to 5 for all the datasets. 
 Secondly, to find a good initial TR decomposition, the tolerance $\epsilon$ (the value of the change in relative error between two adjacent steps) is set to $1e-8$ and the maximum number of iterations $IT$ is set to 100. 
 Note that the performance of online algorithms depends on the quality of the initial decomposition \cite{zhou2016AcceleratingOnline,ma2018RandomizedOnline}. However, exploring the impact of initialization is not our main purpose. So, we use the same initialization for both STR and rSTR. Whereas, in practice, it is better 
 to validate the goodness of the initialization 
 to obtain the best subsequent effectiveness.
 
 In addition, in terms of method-specific parameters, for the six batch algorithms (i.e., the six baselines), the default settings, $\epsilon = 1e-10$ and $IT=50$, are used, and
we adopt the same sketch size $m=1000$ in all the randomized algorithms, which has little 
effect on experiments except for \Cref{fig:rank}, since the rank is changing there. Besides, unless otherwise stated, we always set $t^{new}=5$ for all the datasets. 

\subsection{Effectiveness and efficiency}
\label{ssec:efficient}
The experiments are conducted on two real-world datasets of varying characteristics and 
higher-order structure. Specifically, we extract 360 gray-scale images from the popular image dataset \emph{Columbia Object Image Library (COIL-20)}\footnote{\url{https://cave.cs.columbia.edu/repository/COIL-20}} to form a tensor of size $416\times 448\times 360$ and 300 frames of a popular video sequences from \emph{Hall}\footnote{\url{https://github.com/qbzhao/BRTF/tree/master/videos}} to form a tensor of size $144\times  176\times 3\times 300$.

\begin{figure}[htbp] 
	\centering 
	\subfloat[COIL: $416 \times 448 \times 360$]{\includegraphics[scale=0.3]{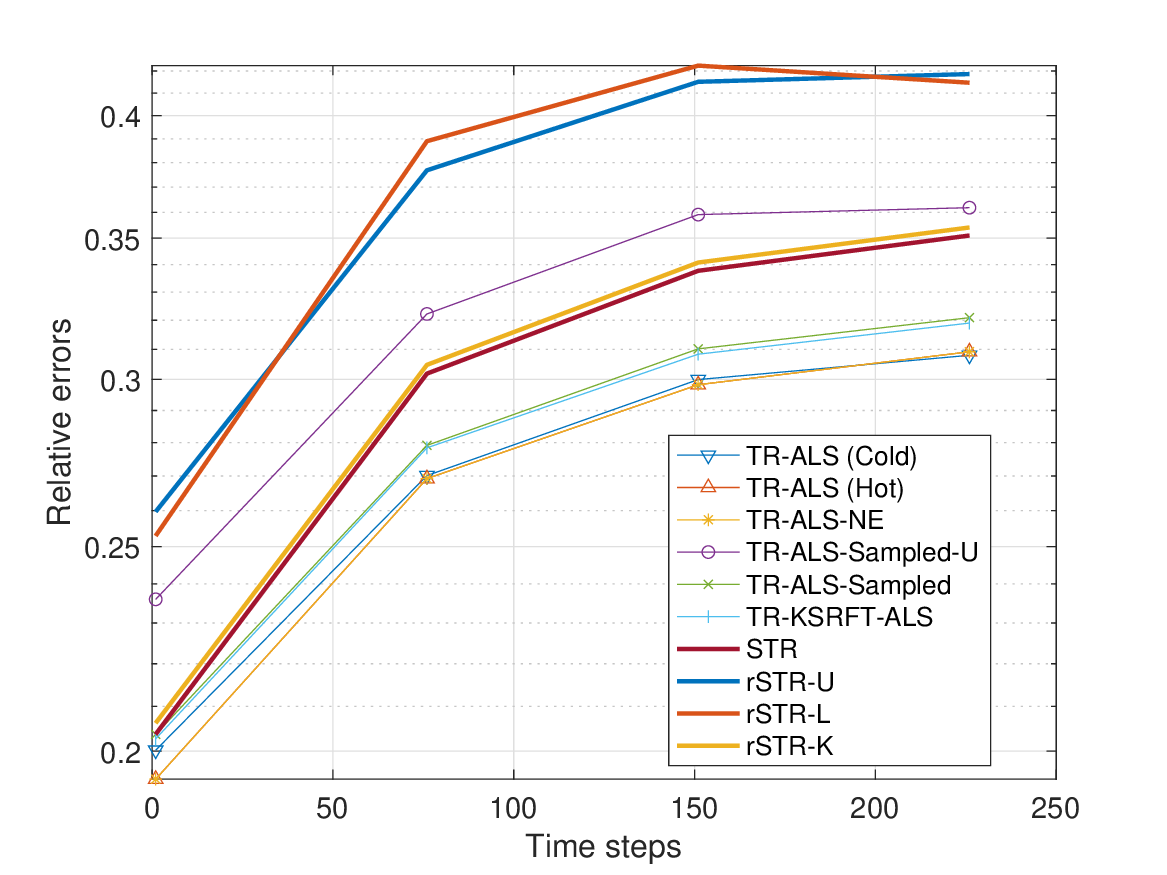}} 
	\subfloat[COIL: $416 \times 448 \times 360$]{\includegraphics[scale=0.3]{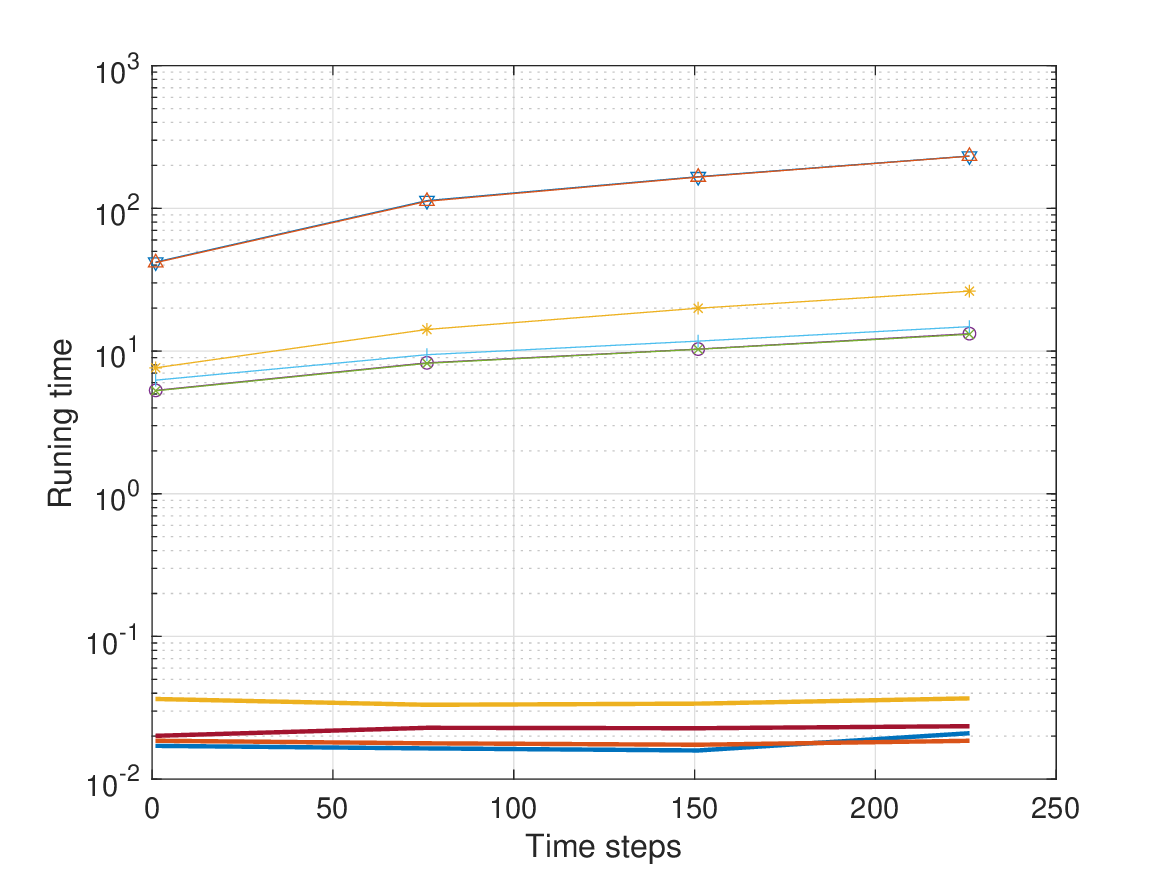}}  
    \quad
	\subfloat[Hall: $144 \times 176 \times 3 \times 300$]{\includegraphics[scale=0.3]{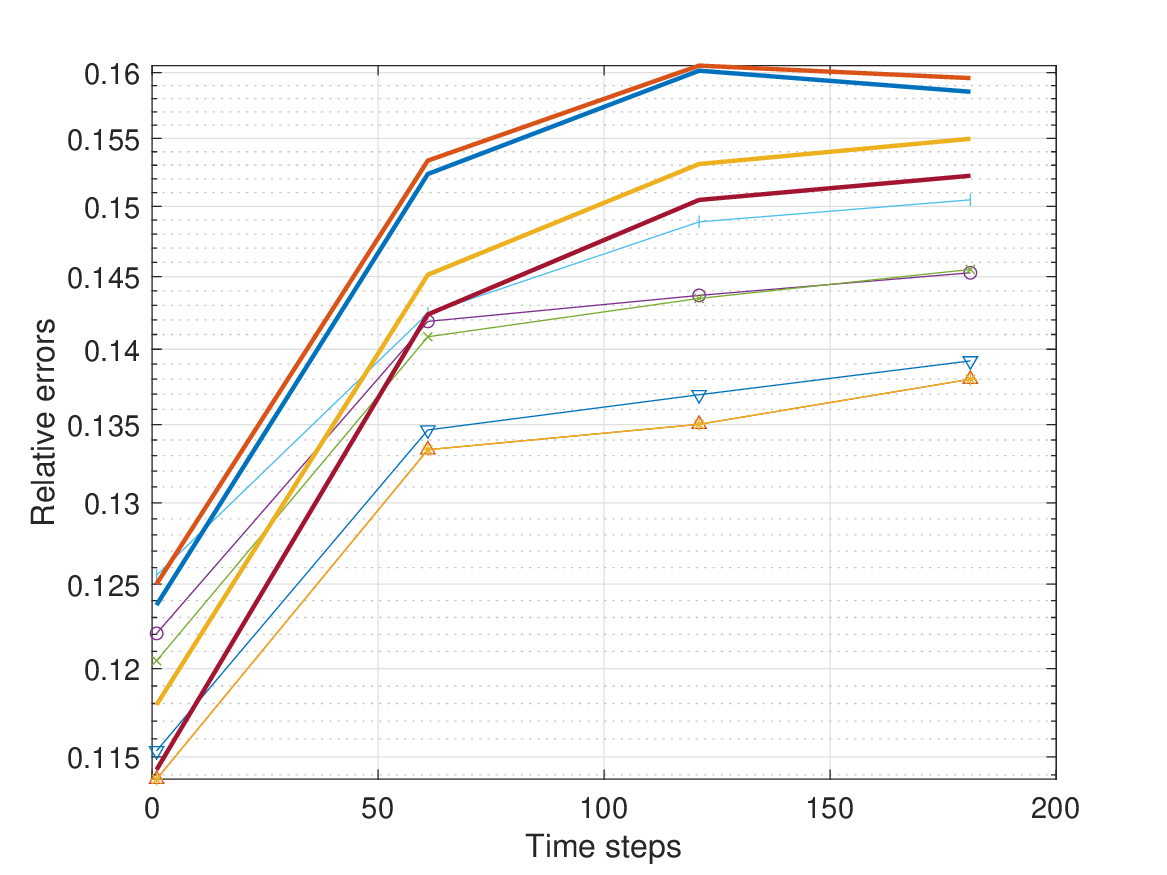}} 
	\subfloat[Hall: $144 \times 176 \times 3 \times 300$]{\includegraphics[scale=0.3]{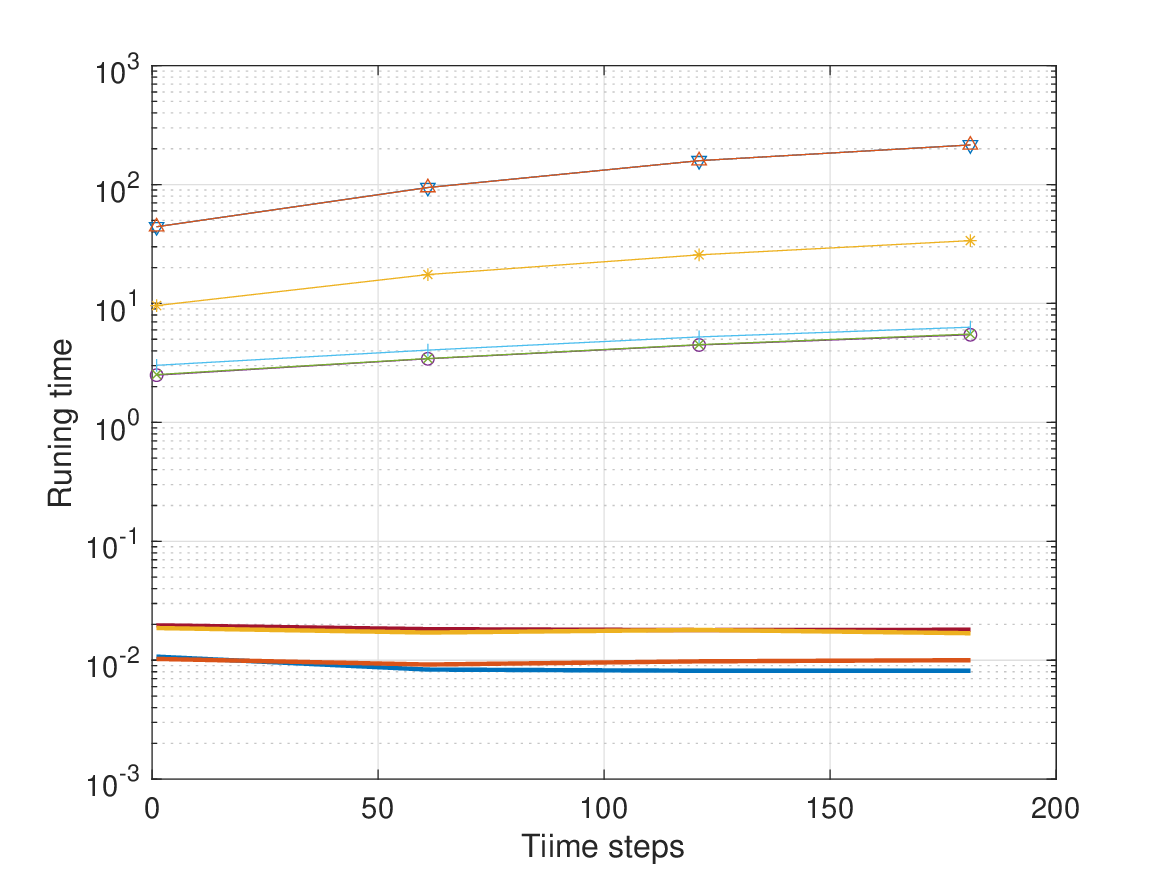}}  
	\caption{Time steps v.s. Relative errors and Time steps v.s. Running time output by algorithms for two real-world datasets.}
	\label{fig:real}
\end{figure}

For these two tensors, the relative errors and processing time for each time step of various algorithms are reported in \Cref{fig:real}, from which we can see that the batch methods, i.e, TR-ALS, TR-ALS-NE, TR-ALS-Sampled-U, TR-ALS-Sampled and TR-KSRFT-ALS, have the expected results which we have known from \cite{yu2022PracticalSketchingBased,yu2022PracticalAlternating}. That is, TR-ALS-NE is identical to TR-ALS in terms of accuracy, but takes much less time due to the structure being used in the algorithm; the three randomized algorithms can accelerate the deterministic methods, however, loss some accuracy.
In addition, TR-ALS (Cold) is slightly less accurate than TR-ALS (Hot). The main reason is that using previous results as initialization can provide a descending seed point for the ALS algorithm, while TR-ALS (Cold) discards this useful information completely.

Our proposed algorithms, i.e., STR and rSTR, show very promising results in terms of accuracy and speed.
Specifically, 
STR is fairly consistent and very similar to the 
batch methods in accuracy; rSTR performs slightly worse but the differences are not remarkable. 
However, both of them are much faster than all the batch methods including the randomized ones. 
Comparing STR and rSTR, the sampling-based rSTR, i.e., rSTR-U and rSTR-L, shows an 
advantage in computing time, however, the projection-based rSTR, i.e., rSTR-K, is not very competitive in this respect. The main reason is that an expensive step, i.e., the mixing tensor step, 
needs to be performed at each time step; see \cite{yu2022PracticalSketchingBased} for more details. 
While, 
when the tensor order increases, the advantage of rSTR-K in running time will gradually emerge; see the experiments in \Cref{ssec:comparison} below. 

\subsection{More comparisons}
\label{ssec:comparison}
In this subsection, with the synthetic data formed by TR decomposition whose TR-cores are generated by random Gaussian tensors with entries drawn independently from a standard normal distribution, we show the impact on performance of various parameters appearing in the computational complexities of algorithms; see \Cref{tab:complexity} for the specific complexities. More specifically, 
we compare each algorithm by varying five parameters: the order ($N$), the dimension ($I$), the dimension of the temporal mode ($I_{N}$), the temporal slice size ($t^{new}$), and the rank ($R_{true} = R$).

We first vary $N$ and $I$. Numerical results on decomposition for four tensors with different  orders and dimensions are given in  \Cref{fig:order}, which shows the similar results to the previous experiments in \Cref{ssec:efficient}. This further validates the effectiveness and efficiency of our algorithms. Note that TR-ALS (Cold) may fluctuate wildly. This is because, as explained above, reinitialization may lead the method to be nonstable. 

\begin{figure}[htbp] 
	\centering 
	\subfloat[$\tensor{X}: 300 \times 300 \times 300$]{\includegraphics[scale=0.3]{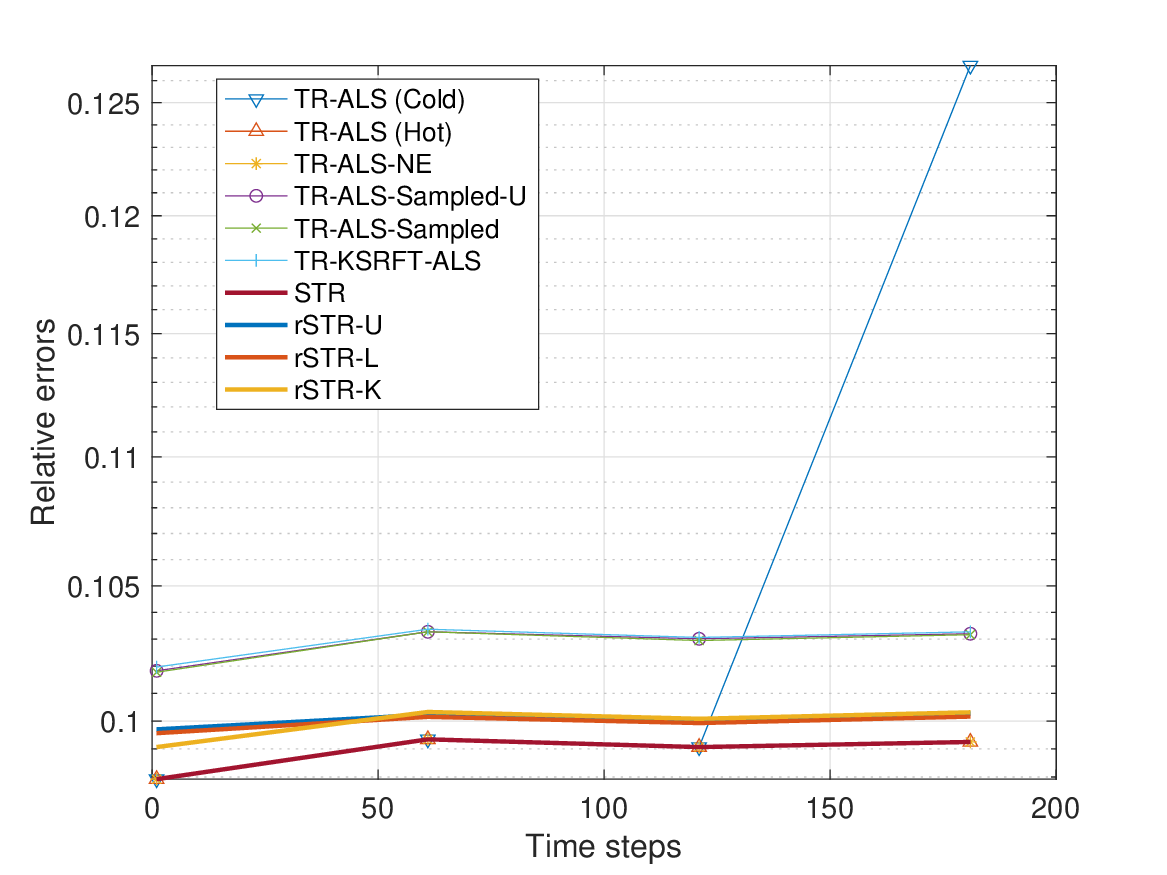}} 
	\subfloat[$\tensor{X}: 300 \times 300 \times 300$]{\includegraphics[scale=0.3]{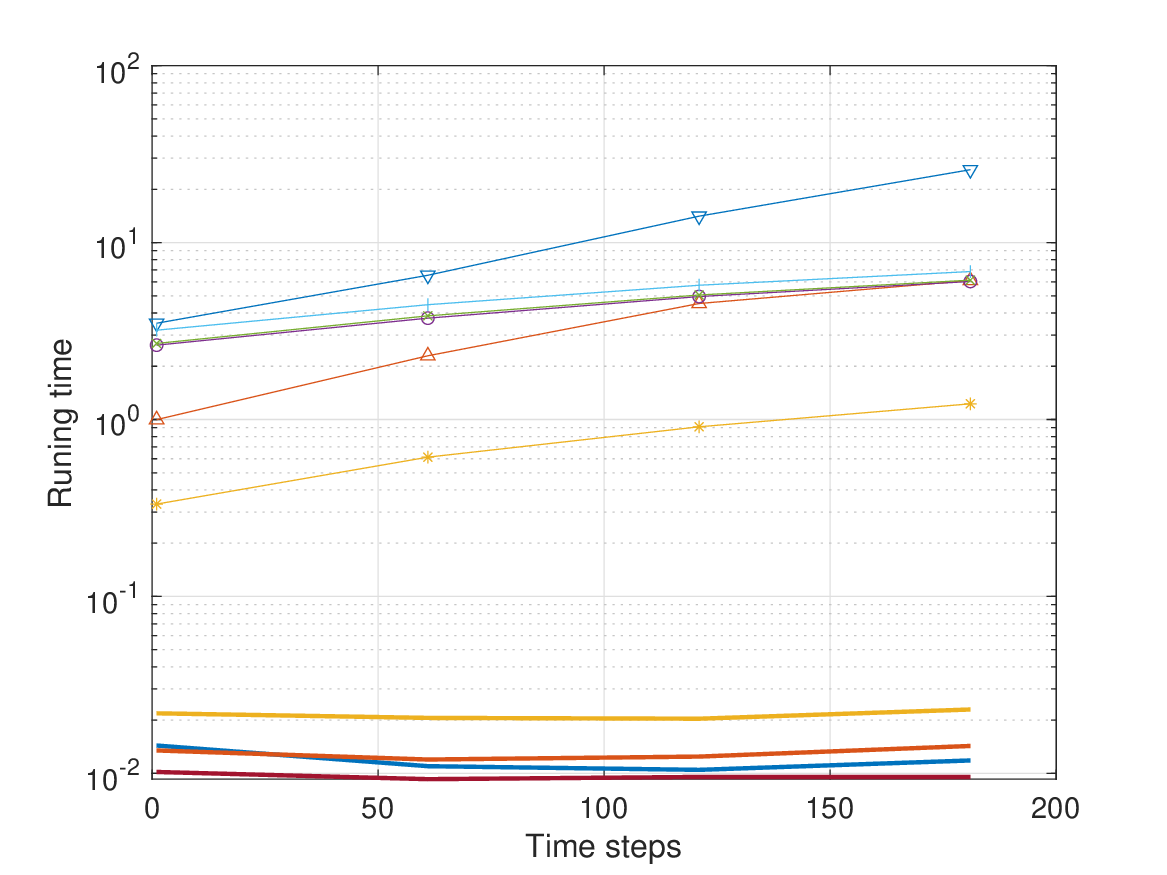}}  
    \quad
	\subfloat[$\tensor{X}: 500 \times 500 \times 500$]{\includegraphics[scale=0.3]{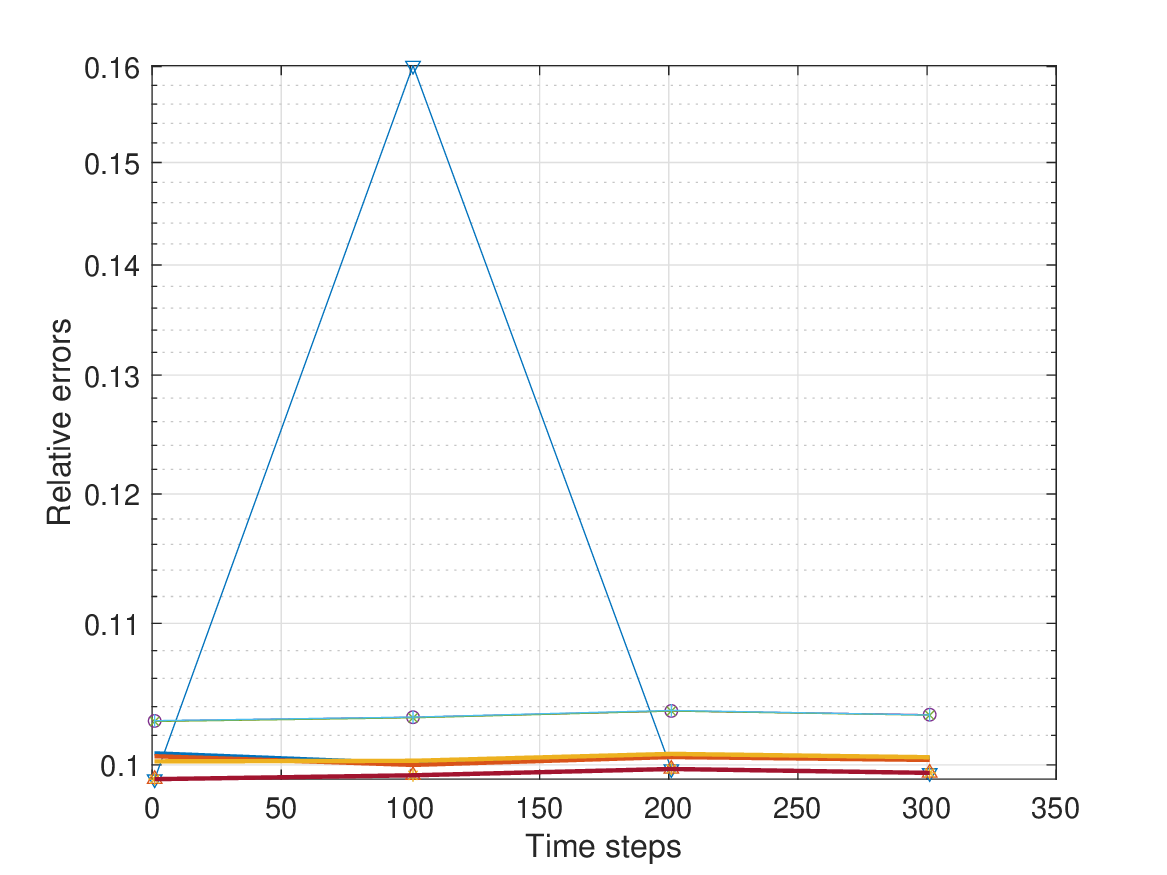}} 
	\subfloat[$\tensor{X}: 500 \times 500 \times 500$]{\includegraphics[scale=0.3]{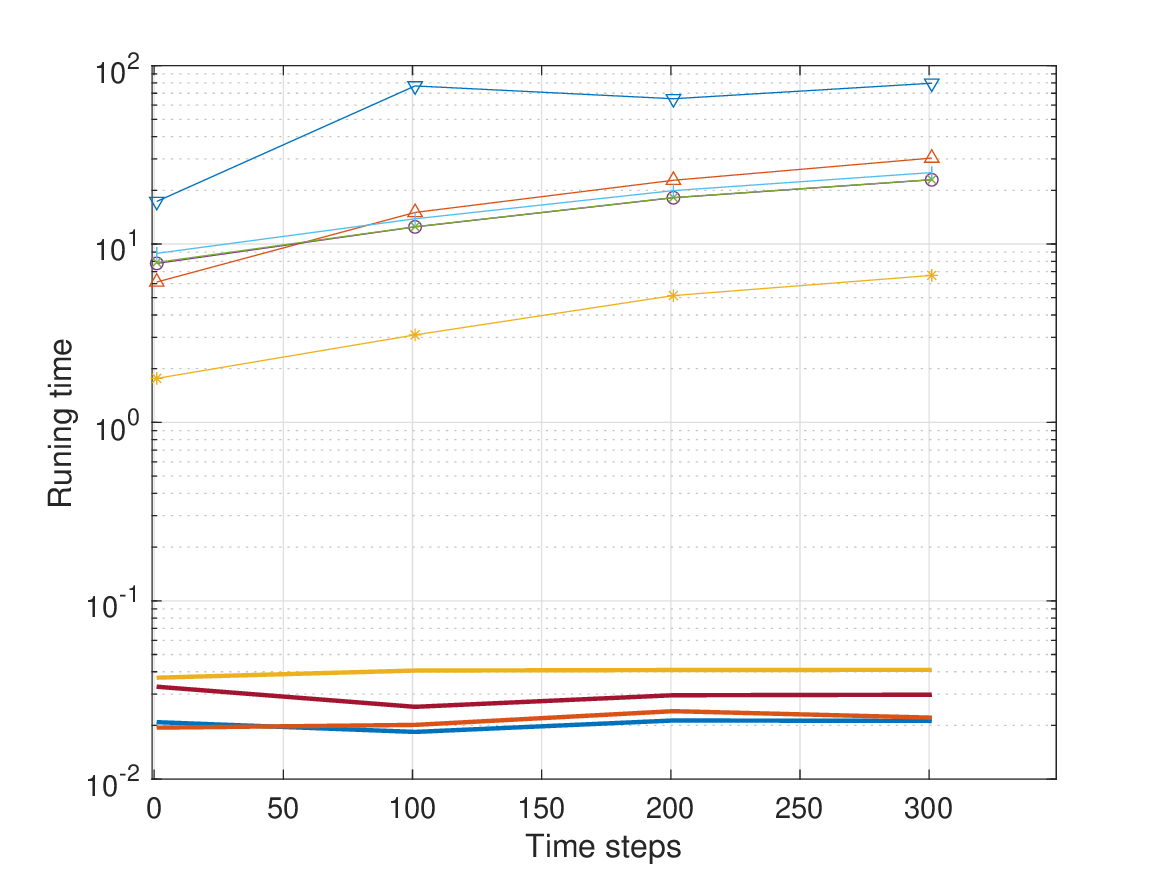}} 
	\quad
	\subfloat[$\tensor{X}: 40 \times 40 \times 40 \times 40 \times 40$]{\includegraphics[scale=0.3]{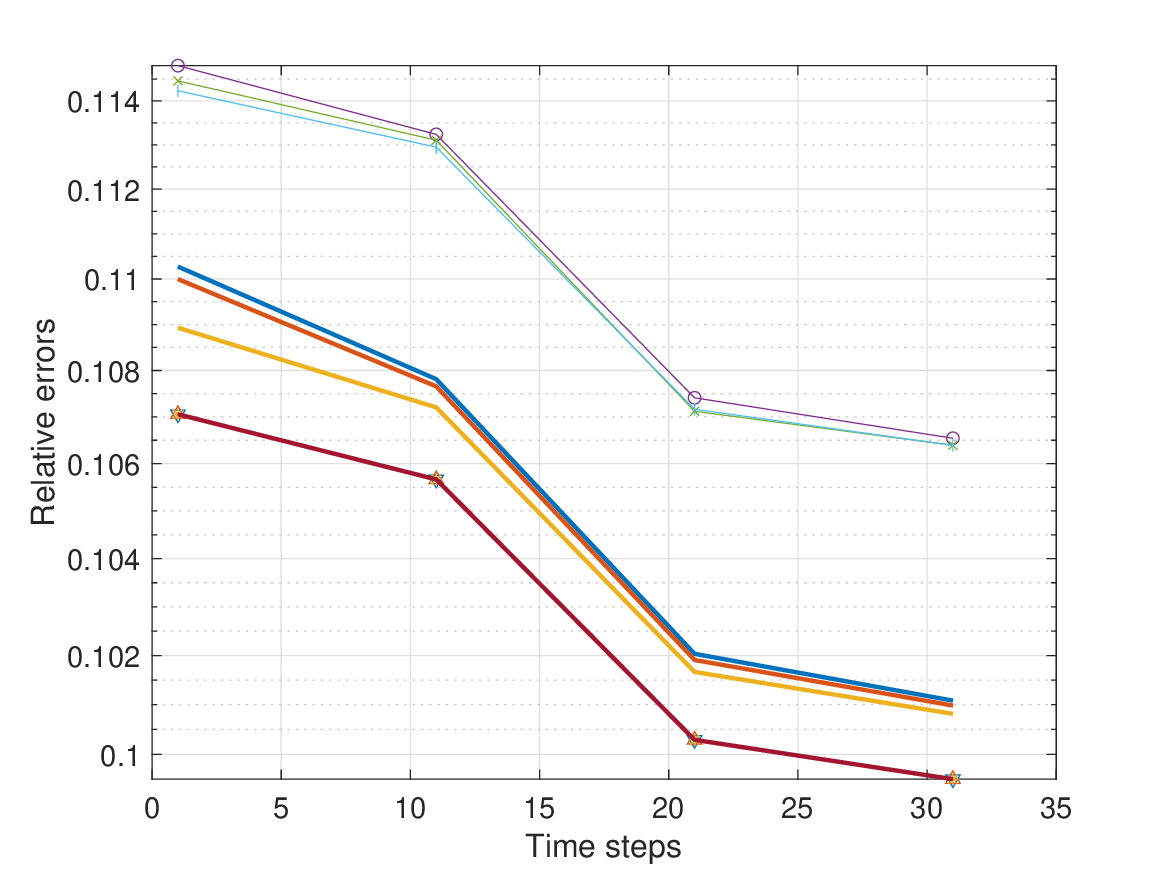}} 
	\subfloat[$\tensor{X}: 40 \times 40 \times 40 \times 40 \times 40$]{\includegraphics[scale=0.3]{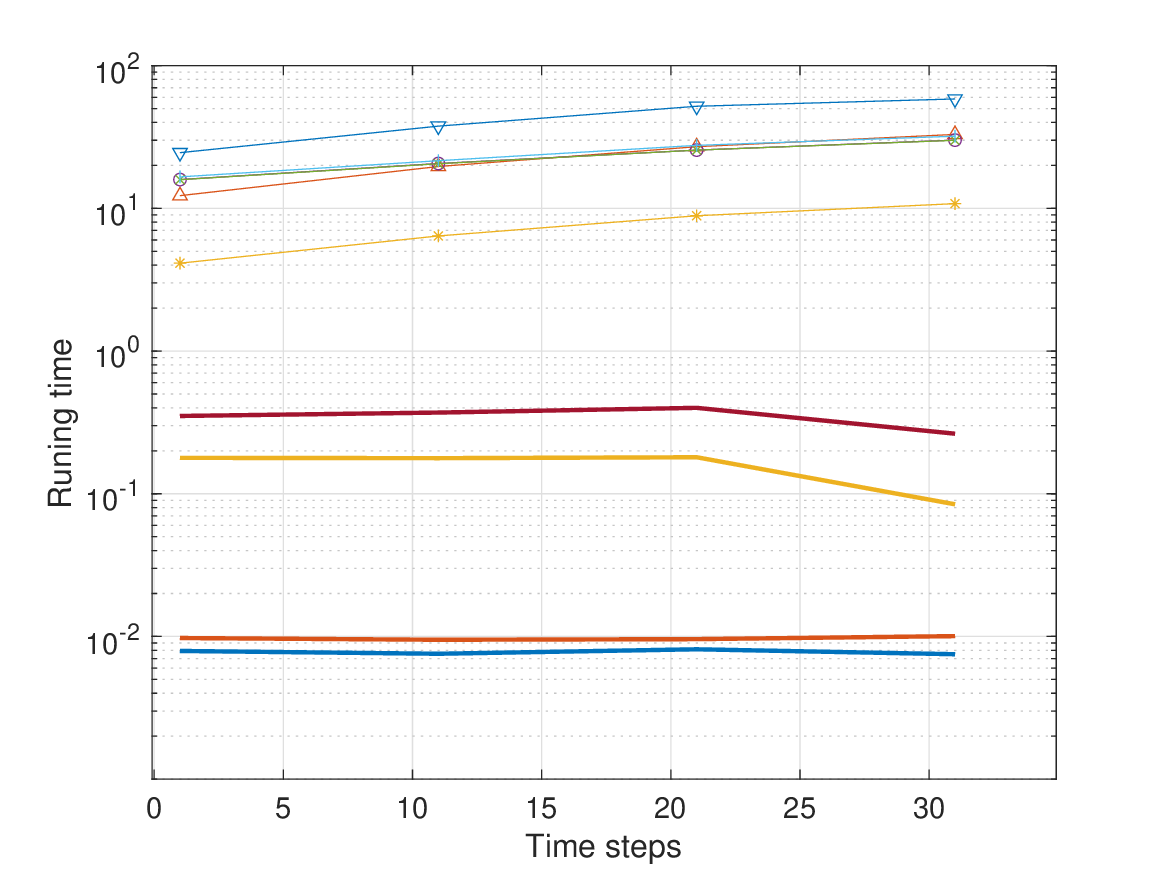}}  
    \quad
	\subfloat[$\tensor{X}: 60 \times 60 \times 60 \times 60 \times 60$]{\includegraphics[scale=0.3]{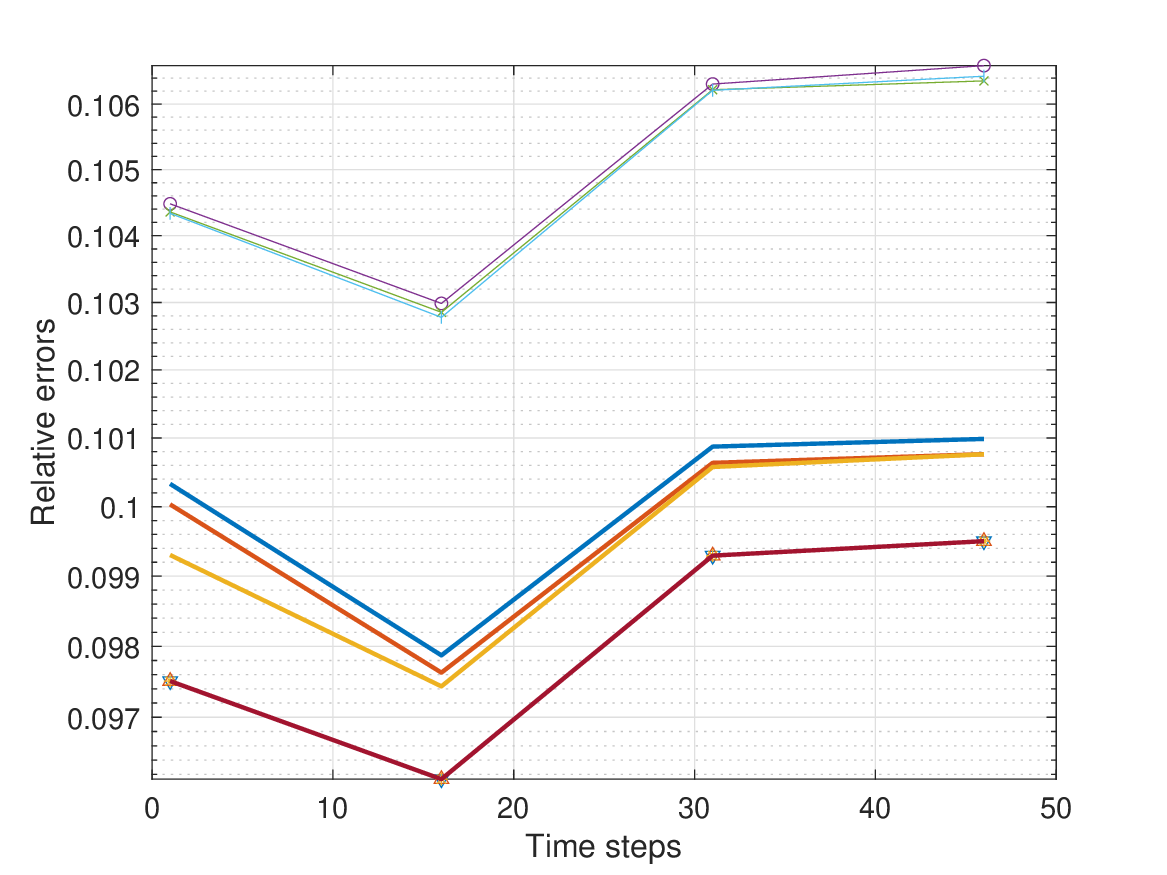}} 
	\subfloat[$\tensor{X}: 60 \times 60 \times 60 \times 60 \times 60$]{\includegraphics[scale=0.3]{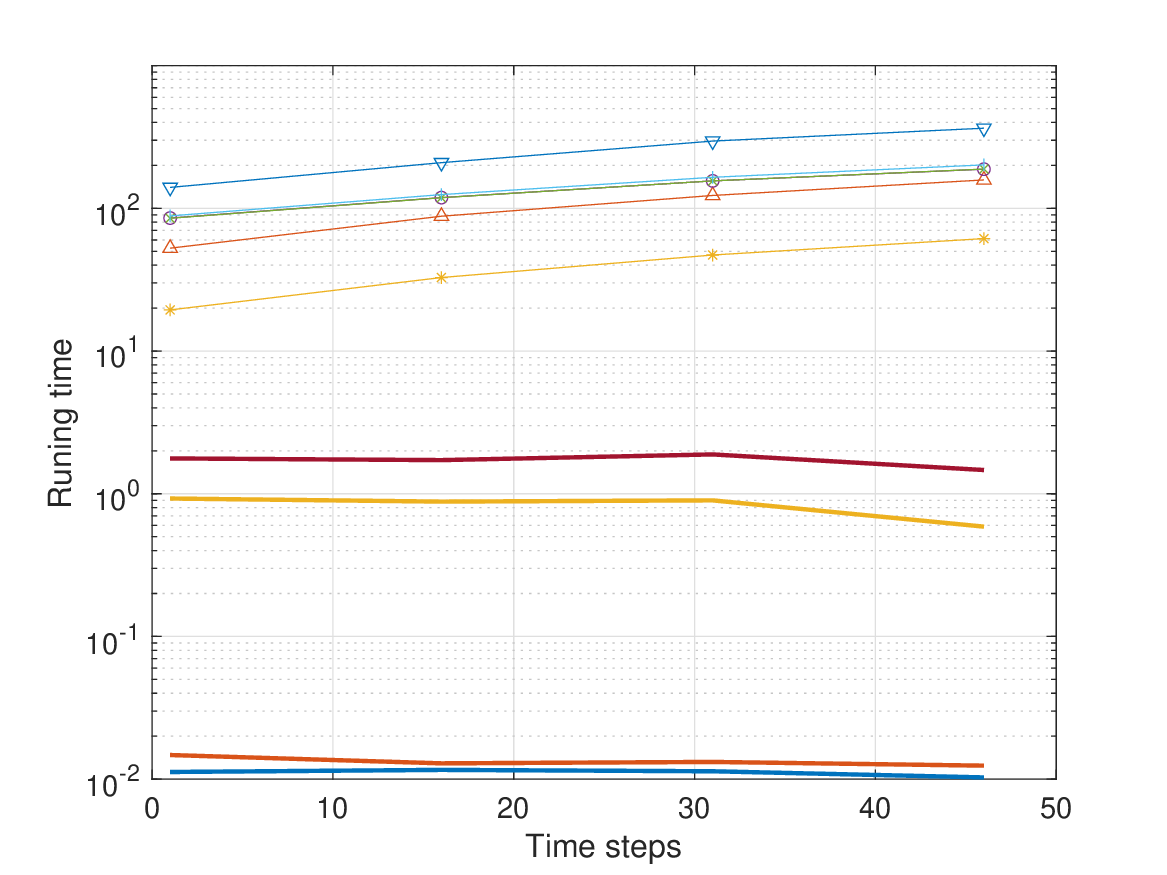}} 
	\caption{Time steps v.s. Relative errors and Time steps v.s. Running time output by algorithms for tensors with different orders and dimensions.}
	\label{fig:order}
\end{figure}

Now, we vary $I_N$. Specifically, we set 
$\tensor{X}: 30 \times 30 \times 30 \times 30 \times I_N$ with $I_N = 120,140,160,180,200$. The final relative errors and the total running time for each tensor are measured and displayed in \Cref{fig:length}, and we can see that, with the errors being almost unchanged, the complexities for all the algorithms increase as 
the length of processed data grows as expected. 
\begin{figure}[htbp] 
	\centering 
	\subfloat[$\tensor{X}: 30 \times 30 \times 30 \times 30 \times I_N$]{\includegraphics[scale=0.31]{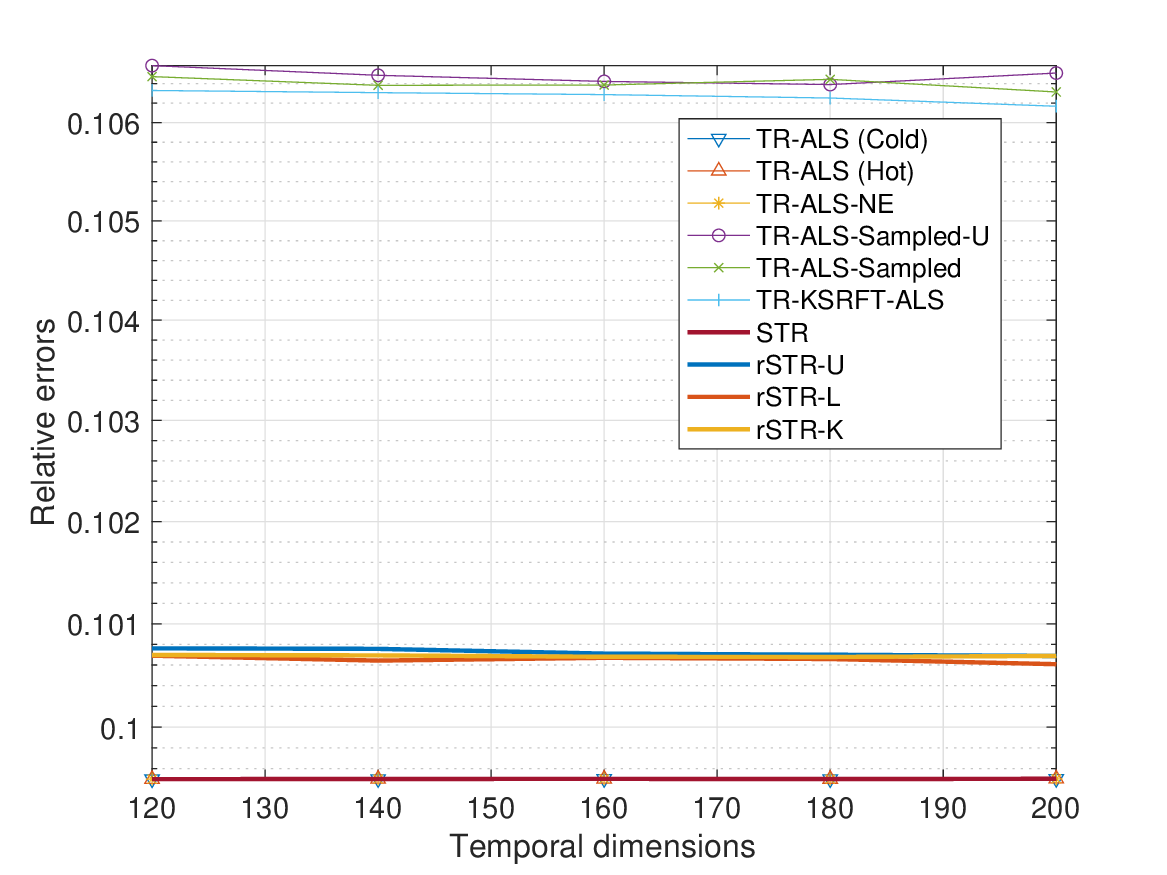}} 
	\subfloat[$\tensor{X}: 30 \times 30 \times 30 \times 30 \times I_N$]{\includegraphics[scale=0.31]{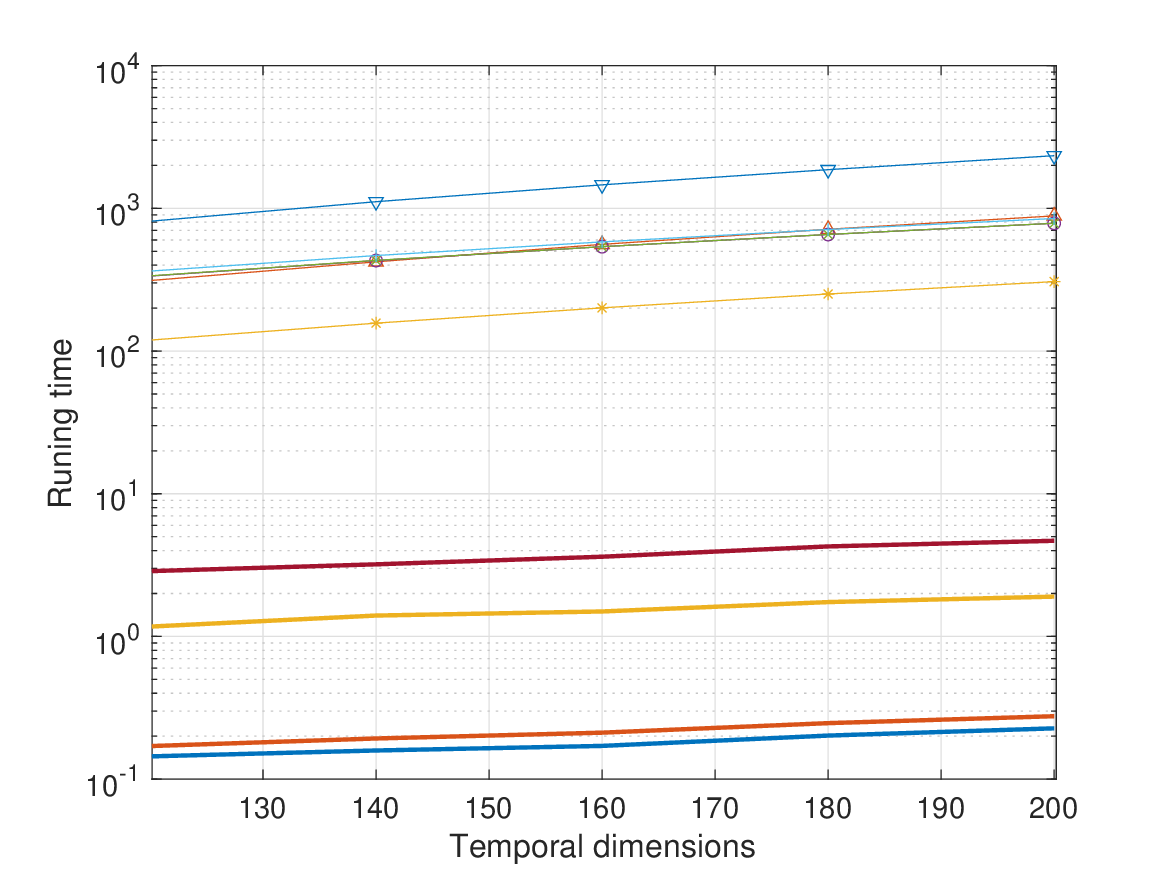}}  
	\caption{Temporal dimensions v.s. Relative errors and Temporal dimensions v.s. Running time output by algorithms for tensors with different temporal dimensions $I_N = 120,140,160,180,200$.}
	\label{fig:length}
\end{figure}

Thirdly, we consider the tensor $\tensor{X}: 30 \times 30 \times 30 \times 30 \times 200$ with $t^{new} = 2,4,6,8,10$ and record the final relative errors and the total running time when the whole tensor is decomposed. 
As can be seen from \Cref{fig:slice}, the temporal slice size has little effect on the quality of the decomposition, but in terms of running time, the larger the temporal slice size is, the less the total running time is. 
This is because, for a tensor with a fixed time dimension, larger temporal slice size means fewer time steps and hence 
less running time.
\begin{figure}[htbp] 
	\centering 
	\subfloat[$\tensor{X}: 30 \times 30 \times 30 \times 30 \times 200$]{\includegraphics[scale=0.31]{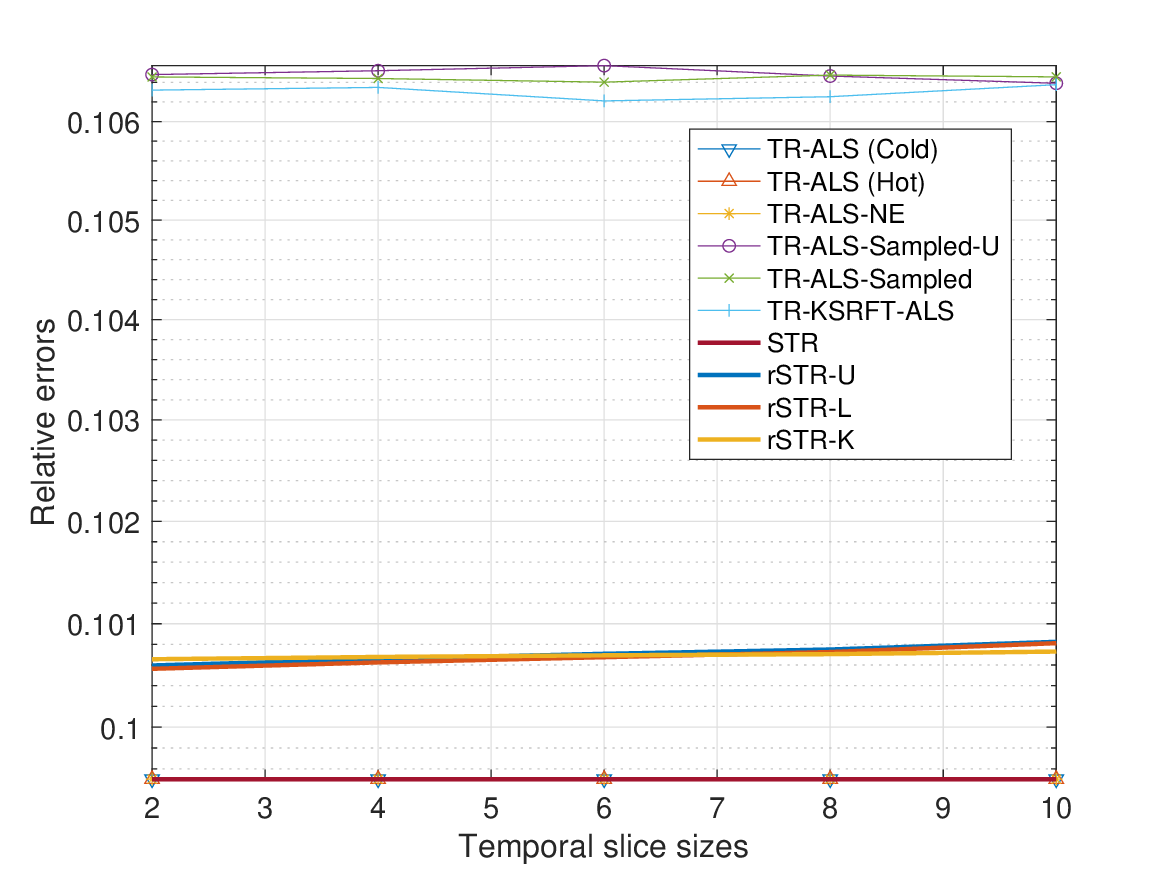}} 
	\subfloat[$\tensor{X}: 30 \times 30 \times 30 \times 30 \times 200$]{\includegraphics[scale=0.31]{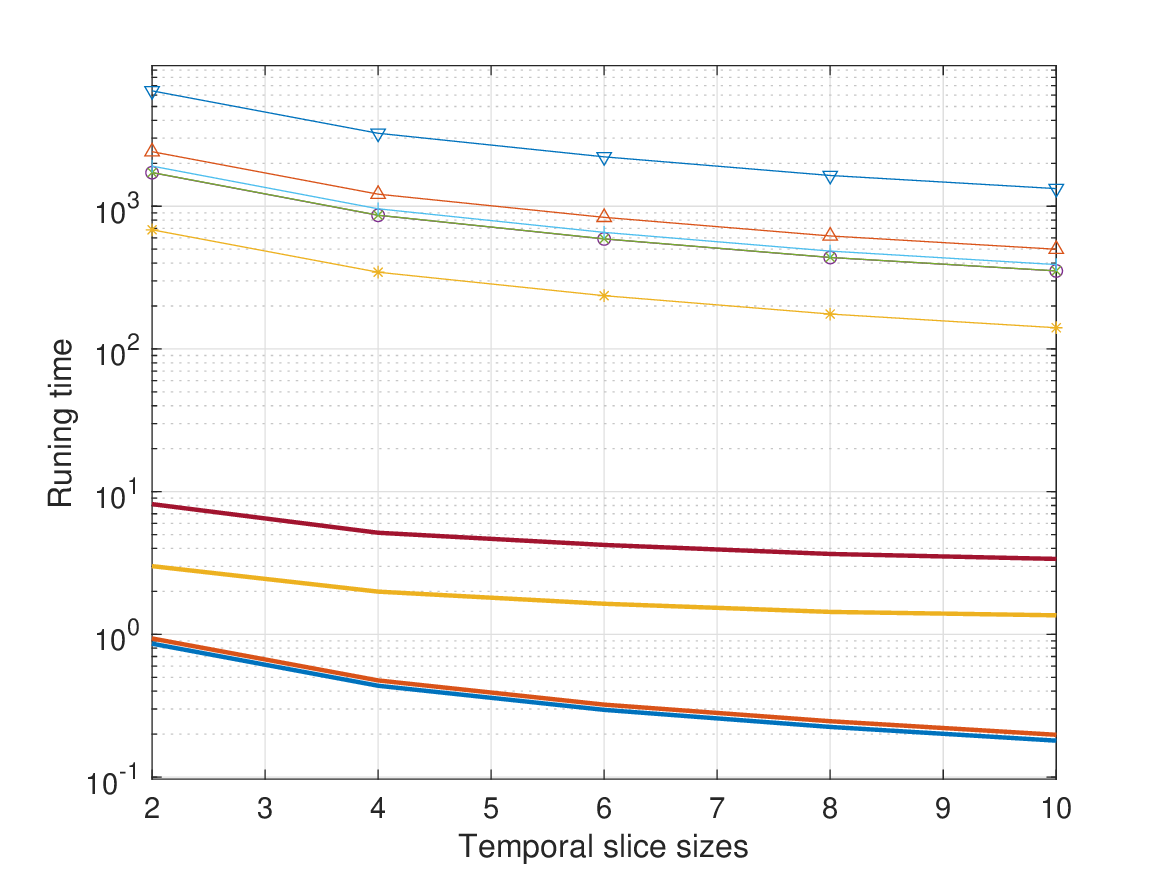}}  
	\caption{Temporal slice sizes v.s. Relative errors and Temporal slice sizes v.s. Running time output by algorithms for a tensor with different temporal slice sizes  $t^{new} = 2,4,6,8,10$.}
	\label{fig:slice}
\end{figure}

Finally, we vary the rank $R_{true}=R$ for tensors 
$\tensor{X}: 30 \times 30 \times 30 \times 30 \times 200$. The final relative errors and the total running time for $R = 3,4,5,6$ are reported in 
\Cref{fig:rank}, from which it is seen that 
the relative errors for STR and three offline deterministic algorithms, i.e., TR-ALS (Cold), TR-ALS (Hot), TR-ALS-NE, have no 
change 
as $R$ varies. Whereas, for 
randomized algorithms, the errors will increase with the rank growing. This is because, by \Cref{thm:uni,thm:lev,thm:ksrft}, the sketch size of randomized algorithms is related to the size of the rank, while the former is fixed in our experiments.
In comparison, 
the range of change for our rSTR is smaller, which is mainly due to the use of the better decomposition results from the previous time step. 
As for the running time, all the randomized algorithms hardly varies as the rank increases, while several deterministic methods increase a little. This is not well reflected in \Cref{tab:complexity} because the assumptions there are not  satisfied when the rank increases, thus making the overall leading order complexity change. 
\begin{figure}[htbp] 
	\centering 
	\subfloat[$\tensor{X}: 30 \times 30 \times 30 \times 30 \times 200$]{\includegraphics[scale=0.31]{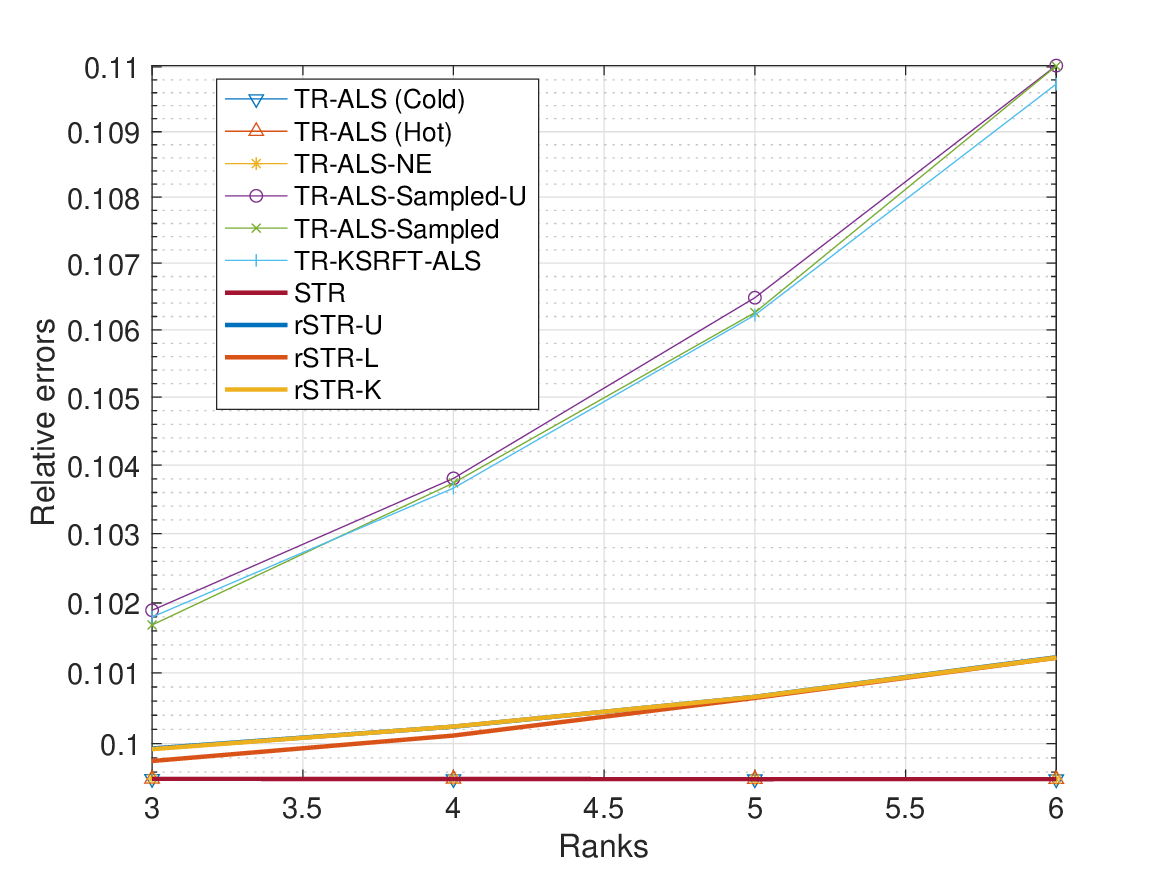}} 
	\subfloat[$\tensor{X}: 30 \times 30 \times 30 \times 30 \times 200$]{\includegraphics[scale=0.31]{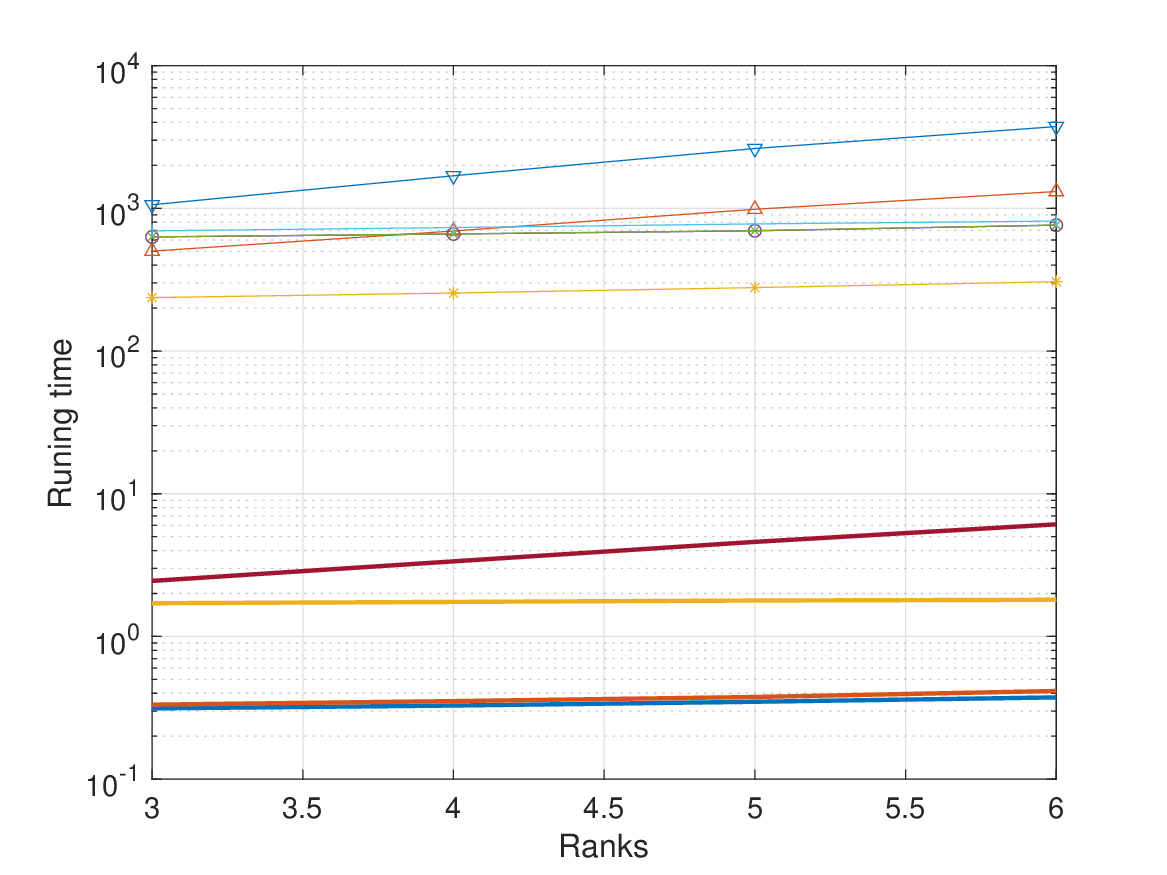}}  
	\caption{Ranks v.s. Relative errors and Ranks v.s. Running time output by algorithms for tensors with different ranks $R_{true} = R = 3,4,5,6$.}
	\label{fig:rank}
\end{figure}

\section{Concluding Remarks}
\label{sec:conclusion}
This paper 
discusses the problem of tracking TR decompositions of streaming tensors. 
A streaming algorithm, i.e., STR, is first proposed that can efficiently monitor the new decomposition by employing complementary TR-cores to temporally store the valuable information from the previous time step.
Then, we provide a randomized variant of STR, i.e., rSTR, which can permit various randomization techniques conveniently and cheaply due to the use of the structure of the coefficient matrices in TR-ALS.
Numerical results on both real-world and synthetic datasets  
demonstrate that our algorithms are comparable to the accurate batch methods in accuracy, and outperform them considerably in terms of computational cost. 

There is some room for methodological improvement. By incorporating numerous popular regularizers and constraints, such as nonnegativity, we can further increase the adaptability of our methods, making them more suited for applications such as computer vision.
Moreover, our algorithms presume that the rank of TR decomposition remains constant throughout the streaming process. Increasing TR-ranks in streaming TR decompositions is a viable option.
In addition, 
it is also valuable to extend our methods to accommodate streaming tensors that can be modified in any mode, i.e., multi-aspect streaming tensors.

\section*{Declarations}

\subsection*{Data Availability}
The data that support the findings of this study are available from the corresponding author upon reasonable request.

\subsection*{Competing Interests}
The authors declare that they have no conflict of interest.


\bibliographystyle{spmpsci}      
\bibliography{BIT-bibliography}   

\begin{thebibliography}{10}
\providecommand{\url}[1]{{#1}}
\providecommand{\urlprefix}{URL }
\expandafter\ifx\csname urlstyle\endcsname\relax
  \providecommand{\doi}[1]{DOI~\discretionary{}{}{}#1}\else
  \providecommand{\doi}{DOI~\discretionary{}{}{}\begingroup
  \urlstyle{rm}\Url}\fi

\bibitem{ahmadi-asl2021CrossTensor}
Ahmadi-Asl, S., Caiafa, C.F., Cichocki, A., Phan, A.H., Tanaka, T., Oseledets,
  I., Wang, J.: Cross tensor approximation methods for compression and
  dimensionality reduction.
\newblock IEEE Access \textbf{9}, 150809--150838 (2021).
\newblock \doi{10.1109/ACCESS.2021.3125069}

\bibitem{ahmadi-asl2020RandomizedAlgorithms}
Ahmadi-Asl, S., Cichocki, A., Phan, A.H., {Asante-Mensah}, M.G., Ghazani, M.M.,
  Tanaka, T., Oseledets, I.V.: Randomized algorithms for fast computation of
  low rank tensor ring model.
\newblock Mach. Learn.: Sci. Technol. \textbf{2}(1), 011001 (2020).
\newblock \doi{10.1088/2632-2153/abad87}

\bibitem{kolda2006TensorToolbox}
Bader, B.W., Kolda, T.G., et~al.: Tensor toolbox for matlab (2021).
\newblock \urlprefix\url{https://www.tensortoolbox.org}.
\newblock Version 3.2.1

\bibitem{battaglino2018PracticalRandomized}
Battaglino, C., Ballard, G., Kolda, T.G.: A practical randomized {CP} tensor
  decomposition.
\newblock SIAM J. Matrix Anal. Appl. \textbf{39}(2), 876--901 (2018).
\newblock \doi{10.1137/17M1112303}

\bibitem{chachlakis2021DynamicL1Norm}
Chachlakis, D.G., Dhanaraj, M., Prater-Bennette, A., Markopoulos, P.P.: Dynamic
  l1-norm tucker tensor decomposition.
\newblock IEEE J. Sel. Topics Signal Process. \textbf{15}(3), 587--602 (2021).
\newblock \doi{10.1109/JSTSP.2021.3058846}

\bibitem{drineas2006FastMonte}
Drineas, P., Kannan, R., Mahoney, M.W.: Fast monte carlo algorithms for
  matrices i: Approximating matrix multiplication.
\newblock SIAM J. Comput. \textbf{36}(1), 132--157 (2006).
\newblock \doi{10.1137/S0097539704442684}

\bibitem{drineas2012FastApproximation}
Drineas, P., Magdon-Ismail, M., Mahoney, M.W., Woodruff, D.P.: Fast
  approximation of matrix coherence and statistical leverage.
\newblock J. Mach. Learn. Res. \textbf{13}(1), 3475--3506 (2012)

\bibitem{drineas2011FasterLeast}
Drineas, P., Mahoney, M.W., Muthukrishnan, S., Sarl{\'o}s, T.: Faster least
  squares approximation.
\newblock Numer. Math. \textbf{117}(2), 219--249 (2011).
\newblock \doi{10.1007/s00211-010-0331-6}

\bibitem{espig2012NoteTensor}
Espig, M., Naraparaju, K.K., Schneider, J.: A note on tensor chain
  approximation.
\newblock Comput. Visual Sci. \textbf{15}, 331–344 (2012).
\newblock \doi{10.1007/s00791-014-0218-7}

\bibitem{he2022PatchTrackingbased}
He, Y., Atia, G.K.: Patch tracking-based streaming tensor ring completion for
  visual data recovery.
\newblock IEEE Trans. Circuits Syst. Video Technol. \textbf{32}(12), 8312--8326
  (2022).
\newblock \doi{10.1109/TCSVT.2022.3190818}

\bibitem{huang2022MultiAspectStreaming}
Huang, Z., Qiu, Y., Yu, J., Zhou, G.: Multi-aspect streaming tensor ring
  completion for dynamic incremental data.
\newblock IEEE Signal Process. Lett. \textbf{29}, 2657--2661 (2022).
\newblock \doi{10.1109/LSP.2022.3231469}

\bibitem{jin2021FasterJohnson}
Jin, R., Kolda, T.G., Ward, R.: Faster {Johnson-Lindenstrauss} transforms via
  {Kronecker} products.
\newblock Inf. Inference \textbf{10}(4), 1533--1562 (2021).
\newblock \doi{10.1093/imaiai/iaaa028}

\bibitem{kolda2009TensorDecompositions}
Kolda, T.G., Bader, B.W.: Tensor decompositions and applications.
\newblock SIAM Rev. \textbf{51}(3), 455--500 (2009).
\newblock \doi{10.1137/07070111X}

\bibitem{kressner2022StreamingTensor}
Kressner, D., Vandereycken, B., Voorhaar, R.: Streaming tensor train
  approximation.
\newblock arXiv preprint arXiv:2208.02600  (2022)

\bibitem{liu2021IncrementalTensorTrain}
Liu, H., Yang, L.T., Guo, Y., Xie, X., Ma, J.: An incremental tensor-train
  decomposition for cyber-physical-social big data.
\newblock IEEE Trans. Big Data \textbf{7}(2), 341--354 (2021).
\newblock \doi{10.1109/TBDATA.2018.2867485}

\bibitem{ma2018RandomizedOnline}
Ma, C., Yang, X., Wang, H.: Randomized online {CP} decomposition.
\newblock In: 2018 Tenth International Conference on Advanced Computational
  Intelligence (ICACI), pp. 414--419. IEEE, Xiamen, China (2018)

\bibitem{malik2022MoreEfficient}
Malik, O.A.: More efficient sampling for tensor decomposition with worst-case
  guarantees.
\newblock In: Proceedings of the 39th International Conference on Machine
  Learning, vol. 162, pp. 14887--14917. PMLR, Virtual Event (2022)

\bibitem{malik2021SamplingBasedMethod}
Malik, O.A., Becker, S.: A sampling-based method for tensor ring decomposition.
\newblock In: Proceedings of the 38th International Conference on Machine
  Learning, vol. 139, pp. 7400--7411. PMLR, Virtual Event (2021)

\bibitem{mickelin2020AlgorithmsComputing}
Mickelin, O., Karaman, S.: On algorithms for and computing with the tensor ring
  decomposition.
\newblock Numer. Linear Algebra Appl. \textbf{27}(3), e2289 (2020).
\newblock \doi{10.1002/nla.2289}

\bibitem{oseledets2011TensorTrainDecomposition}
Oseledets, I.V.: Tensor-train decomposition.
\newblock SIAM J. Sci. Comput. \textbf{33}(5), 2295--2317 (2011).
\newblock \doi{10.1137/090752286}

\bibitem{sun2006StreamsGraphs}
Sun, J., Tao, D., Faloutsos, C.: Beyond streams and graphs: Dynamic tensor
  analysis.
\newblock In: Proceedings of the 12th ACM SIGKDD International Conference on
  Knowledge Discovery and Data Mining, vol. KDD '06, pp. 374--383. Association
  for Computing Machinery, New York, NY, USA (2006)

\bibitem{sun2008IncrementalTensor}
Sun, J., Tao, D., Papadimitriou, S., Yu, P.S., Faloutsos, C.: Incremental
  tensor analysis: Theory and applications.
\newblock ACM Trans. Knowl. Discov. Data \textbf{2}(3), 1556--4681 (2008).
\newblock \doi{10.1145/1409620.1409621}

\bibitem{sun2020LowRankTucker}
Sun, Y., Guo, Y., Luo, C., Tropp, J., Udell, M.: Low-rank tucker approximation
  of a tensor from streaming data.
\newblock SIAM J. Math. Data Sci. \textbf{2}(4), 1123--1150 (2020).
\newblock \doi{10.1137/19M1257718}

\bibitem{thanh2021AdaptiveAlgorithms}
Thanh, L.T., Abed-Meraim, K., Trung, N.L., Boyer, R.: Adaptive algorithms for
  tracking tensor-train decomposition of streaming tensors.
\newblock In: 2020 28th European Signal Processing Conference (EUSIPCO), pp.
  995--999. IEEE, Amsterdam, Netherlands (2021)

\bibitem{thanh2022ContemporaryComprehensive}
Thanh, L.T., Abed-Meraim, K., Trung, N.L., Hafiane, A.: A contemporary and
  comprehensive survey on streaming tensor decomposition.
\newblock IEEE Trans. Knowl. Data Eng. pp. 1--20 (2022).
\newblock \doi{10.1109/TKDE.2022.3230874}

\bibitem{woodruff2014SketchingTool}
Woodruff, D.P.: Sketching as a tool for numerical linear algebra.
\newblock Found. Trends Theor. Comput. Sci. \textbf{10}(1--2), 1--157 (2014).
\newblock \doi{10.1561/0400000060}

\bibitem{xiao2018EOTDEfficient}
Xiao, H., Wang, F., Ma, F., Gao, J.: {eOTD}: An efficient online {Tucker}
  decomposition for higher order tensors.
\newblock In: 2018 IEEE International Conference on Data Mining (ICDM), pp.
  1326--1331 (2018)

\bibitem{yu2022OnlineSubspace}
Yu, J., Zou, T., Zhou, G.: Online subspace learning and imputation by
  tensor-ring decomposition.
\newblock Neural Netw. \textbf{153}, 314--324 (2022).
\newblock \doi{10.1016/j.neunet.2022.05.023}

\bibitem{yu2022PracticalAlternating}
Yu, Y., Li, H.: Practical alternating least squares for tensor ring
  decomposition.
\newblock arXiv preprint arXiv:2210.11362  (2022)

\bibitem{yu2022PracticalSketchingBased}
Yu, Y., Li, H.: Practical sketching-based randomized tensor ring decomposition.
\newblock arXiv preprint arXiv:2209.05647  (2022)

\bibitem{yuan2018HigherdimensionTensor}
Yuan, L., Cao, J., Zhao, X., Wu, Q., Zhao, Q.: Higher-dimension tensor
  completion via low-rank tensor ring decomposition.
\newblock In: 2018 Asia-Pacific Signal and Information Processing Association
  Annual Summit and Conference (APSIPA ASC), pp. 1071--1076. IEEE, Honolulu,
  HI, USA (2018)

\bibitem{yuan2019RandomizedTensor}
Yuan, L., Li, C., Cao, J., Zhao, Q.: Randomized tensor ring decomposition and
  its application to large-scale data reconstruction.
\newblock In: ICASSP 2019 - 2019 IEEE International Conference on Acoustics,
  Speech and Signal Processing (ICASSP), pp. 2127--2131. IEEE, Brighton
  Conference Centre Brighton, U.K. (2019)

\bibitem{zeng2021IncrementalCP}
Zeng, C., Ng, M.K.: Incremental {CP} tensor decomposition by alternating
  minimization method.
\newblock SIAM J. Matrix Anal. Appl. \textbf{42}(2), 832--858 (2021).
\newblock \doi{10.1137/20M1319097}

\bibitem{zhao2016TensorRing}
Zhao, Q., Zhou, G., Xie, S., Zhang, L., Cichocki, A.: Tensor ring
  decomposition.
\newblock arXiv preprint arXiv:1606.05535  (2016)

\bibitem{zhou2016AcceleratingOnline}
Zhou, S., Vinh, N.X., Bailey, J., Jia, Y., Davidson, I.: Accelerating online
  {CP} decompositions for higher order tensors.
\newblock In: Proceedings of the 22nd ACM SIGKDD International Conference on
  Knowledge Discovery and Data Mining, vol. KDD '16, pp. 1375--1384.
  Association for Computing Machinery, New York, NY, USA (2016)

\end{thebibliography}


\appendix\section*{Appendix}
\section{Proofs}
\label{SUPP:proof}
We first state some preliminaries that will be used in the proofs, where \Cref{SUPP:lemma:structural-conds} is a variant of \cite[Lemma 1]{drineas2011FasterLeast} for multiple right hand sides, \Cref{SUPP:lemma:approx-mm} is a part of \cite[Lemma 8]{drineas2006FastMonte}, and \Cref{SUPP:lemma:approx-cur} is from \cite[Theorem 4]{drineas2011FasterLeast}. 
\begin{lemma} 
\label{SUPP:lemma:structural-conds}
	Let $\OPT \defeq \min_{\mat{X}} \| \mat{A} \mat{X} - \mat{Y}\|_F$ with 
    $\mat{A} \in \bb{R}^{I \times R}$ and $I > R$, let $\mat{U} \in \bb{R}^{I \times \rank(\mat{A})}$ contain the left singular vectors of $\mat{A}$, let $\mat{U}^\perp$ be an orthogonal matrix whose columns span the space perpendicular to $\range(\mat{U})$ and define $\mat{Y}^\perp \defeq \mat{U}^\perp (\mat{U}^{\perp})^\intercal \mat{Y}$.
	If $\mat{\Psi} \in \bb{R}^{m \times I}$ 
    satisfies
	\begin{equation} 
	\label{SUPP:eq:cond-1}
		\sigma^2_{\min} (\mat{\Psi} \mat{U}) \geq \frac{1}{\sqrt{2}},
	\end{equation}
	\begin{equation} 
	\label{SUPP:eq:cond-2}
		\| \mat{U}^\intercal \mat{\Psi}^\intercal \mat{\Psi} \mat{Y}^\perp\|_F^2 \leq \frac{\varepsilon}{2} \OPT^2,
	\end{equation}
	for some $\varepsilon \in (0,1)$,
	then 
	\begin{equation*}
		\| \mat{A} \tilde{\mat{X}} - \mat{Y} \|_F \leq (1+\varepsilon)\OPT,
	\end{equation*}
	where $\tilde{\mat{X}} \defeq \arg\min_{\mat{X}} \| \mat{\Psi} \mat{A} \mat{X} - \mat{\Psi} \mat{Y} \|_F$.
\end{lemma}

\begin{lemma} 
\label{SUPP:lemma:approx-mm}
	Let $\mat{A}$ and $\mat{B}$ be matrices with $I$ rows, and let $\vect{q} \in \bb{R}^I$ be a probability distribution satisfying
	\begin{equation*} 
	\label{SUPP:eq:approx-mm-lemma-cond}
		\vect{q}(i) \geq \beta \frac{\|\mat{A}(i, :)\|^2_2}{\| \mat{A} \|_F^2} ~~\text{for all}~~ i \in [I] ~~\text{and some}~~ \beta \in (0, 1].
	\end{equation*} 
	If $\mat{\Psi} \in \bb{R}^{m \times I}$ is a sampling matrix with the probability distribution $\vect{q}$, then
	\begin{equation*}
		\bb{E} \|\mat{A}^\intercal \mat{B} - \mat{A}^\intercal \mat{\Psi}^\intercal \mat{\Psi} \mat{B} \|_F^2 \leq \frac{1}{\beta m} \|\mat{A}\|_F^2 \|\mat{B}\|_F^2.
	\end{equation*}
\end{lemma}

\begin{lemma}
\label{SUPP:lemma:approx-cur}
	Let $\mat{A} \in \bb{R}^{I \times R}$ with $\|\mat{A}\|_2 \leq 1$, and let $\vect{q} \in \bb{R}^I$ be a probability distribution satisfying
	\begin{equation*} 
		\vect{q}(i) \geq \beta \frac{\|\mat{A}(i, :)\|^2_2}{\| \mat{A} \|_F^2} ~~\text{for all}~~ i \in [I] ~~\text{and some}~~ \beta \in (0, 1].
	\end{equation*} 
	If $\mat{\Psi} \in \bb{R}^{m \times I}$ is a sampling matrix with the probability distribution $\vect{q}$,  $\varepsilon \in (0,1)$ is an accuracy parameter, $\|\mat{A}\|_F^2 \geq \frac{1}{24}$, and 
	\begin{equation*}
		m \geq \frac{96 \|\mat{A}\|_F^2}{\beta \varepsilon^2}\ln \left( \frac{96 \|\mat{A}\|_F^2}{\beta \varepsilon^2 \sqrt{\delta}} \right),
	\end{equation*}
	then, with a probability of at least $1-\delta$,
	\begin{equation*}
		\|\mat{A}^\intercal \mat{A} - \mat{A}^\intercal \mat{\Psi}^\intercal \mat{\Psi} \mat{A} \|_F^2 \leq \varepsilon.
	\end{equation*}
\end{lemma}

\subsection{Proof of \Cref{thm:uni} }
We first state a theorem similar to \cite[Theorem 7]{malik2021SamplingBasedMethod}, i.e., the theoretical guarantee of uniform sampling for TR-ALS.

\begin{theorem}
\label{SUPP:thm:uni-offline}
    Let $\mat{\Psi}_{n}$ be a uniform sampling matrix defined as in \eqref{eq:sampDS}, and
    \begin{equation*}
        \tilde{\mat{G}}_{n(2)} \defeq \arg\min_{\mat{G}_{n(2)}} \left\| \mat{X}_{[n]} \mat{\Psi}_{n}^\intercal - \mat{G}_{n(2)}(\mat{\Psi}_{n}\mat{G}_{[2]}^{\ne n})^\intercal \right\|_F.
    \end{equation*}
   If 
	\begin{equation*}
		m \geq \left( \frac{2 \gamma R_{n} R_{n+1}}{\varepsilon} \right) \max \left[ \frac{48}{\varepsilon}\ln \left( \frac{96 \gamma R_{n} R_{n+1}}{ \varepsilon^2 \sqrt{\delta}} \right), \frac{1}{\delta} \right],  
	\end{equation*}
	with  $\varepsilon \in (0,1)$, $\delta \in (0,1)$, and $\gamma>1$, then the following inequality holds with a probability of at least $1-\delta$:
	\begin{equation*}
		\left\| \mat{X}_{[n]}  - \tilde{\mat{G}}_{n(2)}(\mat{G}_{[2]}^{\ne n})^\intercal \right\|_F \leq (1+\varepsilon) \min_{\mat{G}_{n(2)}} \left\| \mat{X}_{[n]} - \mat{G}_{n(2)}(\mat{G}_{[2]}^{\ne n})^\intercal \right\|_F.
	\end{equation*}
\end{theorem}

\begin{proof}
Let $\mat{U} \in \bb{R}^{J_n \times \rank(\mat{G}_{[2]}^{\ne n})}$ contain the left singular vectors of $\mat{G}_{[2]}^{\ne n}$ and $\rank(\mat{G}_{[2]}^{\ne n}) = R_{n} R_{n+1}$.
Then, there is a $\gamma >1$ such that 
\begin{equation}
\label{SUPP:eq:row-norm}
	\|\mat{U}(i,:)\|_2^2 \leq \frac{\gamma R_{n} R_{n+1}}{J_n} ~~\text{for all}~~ i \in \left[J_n \right].
\end{equation}
Note that $\|\mat{U}\|_F = \sqrt{R_{n} R_{n+1}}$. Thus, setting $\beta = \frac{1}{\gamma}$, we have 
\begin{equation}
\label{SUPP:eq:prob}
	\frac{1}{J_n} \geq \beta \frac{\|\mat{U}(i,:)\|_2^2}{\|\mat{U}\|_F^2}.
\end{equation}
That is, the uniform probability distribution $\vect{q}$ on $[J_n]$ satisfies \eqref{SUPP:eq:prob}.
Moreover, it is easy to see that
$\|\mat{U}\|_2 = 1\leq 1$, $\|\mat{U}\|_F^2 = R_{n} R_{n+1}>\frac{1}{24}$, and 
\begin{equation*}
	m \geq \frac{96R_{n} R_{n+1}}{\beta \varepsilon^2}\ln \left( \frac{96 R_{n} R_{n+1}}{\beta \varepsilon^2 \sqrt{\delta_1}} \right).
\end{equation*}
Thus, noting that $\mat{\Psi}_{n}$ is a sampling matrix with the probability distribution $\vect{q}$, applying \Cref{SUPP:lemma:approx-cur}  implies that
\begin{equation*}
	\|\mat{U}^\intercal \mat{U} - \mat{U}^\intercal \mat{\Psi}_{n}^\intercal \mat{\Psi}_{n} \mat{U} \|_2 \leq \varepsilon.
\end{equation*}
On the other hand, note that for all $i \in [R_{n} R_{n+1}]$,
\begin{align*}
	|1-\sigma^2_{i}(\mat{\Psi}_{n}\mat{U})| 
	&= |\sigma_{i}(\mat{U}^\intercal \mat{U}) - \sigma_{i}(\mat{U}^\intercal\mat{\Psi}_{n}^\intercal \mat{\Psi}_{n} \mat{U})| \\
	&\leq \|\mat{U}^\intercal \mat{U} - \mat{U}^\intercal \mat{\Psi}_{n}^\intercal \mat{\Psi}_{n} \mat{U} \|_2.
\end{align*}
Thus, choosing $\varepsilon = 1-1/\sqrt{2}$ gives that $\sigma^2_{\min} (\mat{\Psi}_{n} \mat{U}) \geq \frac{1}{\sqrt{2}}$, therefore \eqref{SUPP:eq:cond-1} is satisfied.

Next, we check \eqref{SUPP:eq:cond-2}. Recall that $(\mat{X}_{[n]}^\intercal)^\perp \defeq \mat{U}^\perp (\mat{U}^{\perp})^\intercal \mat{X}_{[n]}^\intercal$. 
Hence, $\mat{U}^\intercal (\mat{X}_{[n]}^\intercal)^\perp =0$ and 
\begin{equation*}
	\| (\mat{\Psi}_{n}\mat{U})^\intercal \mat{\Psi}_{n} (\mat{X}_{[n]}^\intercal)^\perp \|_2^2 = \| \mat{U}^\intercal \mat{\Psi}_{n}^\intercal  \mat{\Psi}_{n} (\mat{X}_{[n]}^\intercal)^\perp - \mat{U}^\intercal (\mat{X}_{[n]}^\intercal)^\perp \|_2^2.
\end{equation*}
Thus, noting \eqref{SUPP:eq:row-norm} and \eqref{SUPP:eq:prob}, applying \Cref{SUPP:lemma:approx-mm}, we get 
\begin{equation*}
	\bb{E} \left[ \| (\mat{\Psi}_{n}\mat{U})^\intercal \mat{\Psi}_{n} (\mat{X}_{[n]}^\intercal)^\perp \|_2^2 \right] \leq \frac{1}{\beta m} \| \mat{U} \|_F^2 \| (\mat{X}_{[n]}^\intercal)^\perp \|_2^2 = \frac{R_{n} R_{n+1} \OPT^2}{\beta m},
\end{equation*}
where $\OPT = \min_{\mat{G}_{n(2)}} \left\| \mat{G}_{[2]}^{\ne n} \mat{G}_{n(2)}^\intercal - \mat{X}_{[n]}^\intercal \right\|_F$.
Markov’s inequality now implies that with probability at least $1-\delta_2$
\begin{equation*}
	\| (\mat{\Psi}_{n}\mat{U})^\intercal \mat{\Psi}_{n} (\mat{X}_{[n]}^\intercal)^\perp \|_2^2 \leq \frac{ R_{n} R_{n+1} \OPT^2}{\delta_2 \beta m}.
\end{equation*}
Setting $m \geq \frac{2R_{n} R_{n+1}}{\delta_2 \beta \varepsilon}$ and using the value of $\beta$ specified above, we have that \eqref{SUPP:eq:cond-2} is indeed satisfied.

Finally, using \Cref{SUPP:lemma:structural-conds} concludes the proof of the theorem.
\qed
\end{proof}

\emph{Proof\ of\ \Cref{thm:uni}}
 For the temporal mode $N$, if 
 \begin{align*}
     \tilde{m}_1 &\geq \left( \frac{2 \gamma R_{N} R_{1}}{\tilde{\varepsilon}_1} \right) \max \left[ \frac{48}{\tilde{\varepsilon}_1}\ln \left( \frac{96 \gamma R_{N} R_{1}}{ \tilde{\varepsilon}_1^2 \sqrt{\delta}} \right), \frac{1}{\delta} \right],\\  \tilde{m}_2 &\geq \left( \frac{2 \gamma R_{N} R_{1}}{\tilde{\varepsilon}_2} \right) \max \left[ \frac{48}{\tilde{\varepsilon}_2}\ln \left( \frac{96 \gamma R_{N} R_{1}}{ \tilde{\varepsilon}_2^2 \sqrt{\delta}} \right), \frac{1}{\delta} \right],\\
     &\cdots
 \end{align*} 
 according to \Cref{SUPP:thm:uni-offline},
at each time step we can obtain a corresponding upper error bound between the new coming tensor and its decomposition as follows
	\begin{align*}
		\text{The 1st time step:} & \left\| \mat{X}_{[N]}^{old}  - \tilde{\mat{G}}_{N(2)}^{old}(\mat{G}_{[2]}^{\ne N})^\intercal \right\|_F \leq (1+\tilde{\varepsilon}_1) \min_{\mat{G}_{N(2)}^{old}} \left\| \mat{X}_{[N]}^{old} - \mat{G}_{N(2)}^{old}(\mat{G}_{[2]}^{\ne N})^\intercal \right\|_F,\\
		\text{The 2nd time step:} & \left\| \mat{X}_{[N]}^{new}  - \tilde{\mat{G}}_{N(2)}^{new}(\mat{G}_{[2]}^{\ne N})^\intercal \right\|_F \leq (1+\tilde{\varepsilon}_2) \min_{\mat{G}_{N(2)}^{new}} \left\| \mat{X}_{[N]}^{new} - \mat{G}_{N(2)}^{new}(\mat{G}_{[2]}^{\ne N})^\intercal \right\|_F,\\
       &  \cdots 
	\end{align*}
        To obtain an upper bound on the error for all current time steps, 
	let $\tilde{\varepsilon} = \min\{\tilde{\varepsilon}_1, \tilde{\varepsilon}_2, \cdots\}$, then the following holds with a probability of at least $1-\delta$:
	\begin{equation*}
		\left\| \mat{X}_{[N]}  - \tilde{\mat{G}}_{N(2)}(\mat{G}_{[2]}^{\ne N})^\intercal \right\|_F \leq (1+\tilde{\varepsilon}) \min_{\mat{G}_{N(2)}} \left\| \mat{X}_{[N]} - \mat{G}_{N(2)}(\mat{G}_{[2]}^{\ne N})^\intercal \right\|_F,
	\end{equation*}
  for 
  $$\tilde{m} \geq \left( \frac{2 \gamma R_{N} R_{1}}{\tilde{\varepsilon}} \right) \max \left[ \frac{48}{\tilde{\varepsilon}}\ln \left( \frac{96 \gamma R_{N} R_{1}}{ \tilde{\varepsilon}^2 \sqrt{\delta}} \right), \frac{1}{\delta} \right].$$
 
	For the non-temporal mode $n$, 
    if 
    $$m'_n \geq \left( \frac{2 \gamma R_{n} R_{n+1}}{\varepsilon'_n} \right) \max \left[ \frac{48}{\varepsilon'_n}\ln \left( \frac{96 \gamma R_{n} R_{n+1}}{ {\varepsilon'}_n^2 \sqrt{\delta}} \right), \frac{1}{\delta} \right],$$ 
    we have
	\begin{equation*}
		\left\| \mat{X}_{[n]}  - \tilde{\mat{G}}_{n(2)}(\mat{G}_{[2]}^{\ne n})^\intercal \right\|_F \leq (1+\varepsilon'_n) \min_{\mat{G}_{n(2)}} \left\| \mat{X}_{[n]} - \mat{G}_{n(2)}(\mat{G}_{[2]}^{\ne n})^\intercal \right\|_F.
	\end{equation*}
	Thus, setting $\varepsilon = \min\{\tilde{\varepsilon}, \varepsilon'_1, \varepsilon'_2, \cdots, \varepsilon'_{N-1}\}$, the proof can be completed.
 \qed

Along the same line, the proofs of \Cref{thm:lev,thm:ksrft} can be completed by using \cite[Theorem 7]{malik2021SamplingBasedMethod} and \cite[Theorem 5]{yu2022PracticalSketchingBased}, repectively.

\section{Specific Algorithms Based on Different Sketches}
\label{SUPP:alg}

\begin{breakablealgorithm}\label{alg:rstr_uni}
\caption{Randomized streaming TR decomposition with uniform sampling (rSTR-U)}
\textbf{Input:} Initial tensor $\tensor{X}^{init}$, TR-ranks $R_1, \cdots, R_N$, new data tensor $\tensor{X}^{new}$ and sampling size $m$

\textbf{Output:} TR-cores $\{ \tensor{G}_n \in \bb{R}^{R_n \times I_n \times R_{n+1}} \}_{n=1}^N$
\begin{algorithmic}[1]
	\Statex // \texttt{Initialization stage}
	\State Using TR-ALS-Sampled (uniform) or other algorithms to decompose $\tensor{X}^{init}$ into TR-cores $\tensor{G}_{1}, \cdots, \tensor{G}_{N}$
	\For{$k = 1, \cdots, N$}
		\State $\texttt{idxs}(:,k) \leftarrow \textsc{Randsample}(I_k, m)$
		\State $(\tensor{G}_{k})_{S} \leftarrow \tensor{G}_{k}(:, \texttt{idxs}(:,k), :)$
	\EndFor
	\For{$n = 1, \cdots, N-1$}
	\State Let $\tensor{G}^{\ne n}_S$ be a tensor of size $R_{n+1} \times m \times R_{n+1}$, where every lateral slice is an $R_{n+1} \times R_{n+1}$ identity matrix
	\For{$k = n+1, \cdots, N, 1, \cdots, n-1$}
		\State $\tensor{G}^{\ne n}_S \leftarrow \tensor{G}^{\ne n}_S \boxast_2 (\tensor{G}_{k})_{S}$
	\EndFor
	\State $\mat{X}_{S[n]}^{init} \leftarrow \textsc{Mode-n-Unfolding}(\tensor{X}(\texttt{idxs}(:,1), \cdots, \texttt{idxs}(:,n-1),:,\texttt{idxs}(:,n+1),\cdots \texttt{idxs}(:,N)))$
	\EndFor
	\For{$n = 1, \cdots, N-1$}	
	\State $\mat{P}_{n} \leftarrow \mat{X}_{S[n]}^{init} \mat{G}^{\ne n}_{S[2]}$
	\State $\mat{Q}_{n} \leftarrow \left( \mat{G}^{\ne n}_{S[2]} \right)^\intercal \mat{G}^{\ne n}_{S[2]}$
	\EndFor
	\For{$\tau = 1, \cdots, t$ time steps}
	\Statex // \texttt{Update stage for temporal mode}
	\State Let $\tensor{G}^{\ne N}_S$ be a tensor of size $R_{1} \times m \times R_{1}$, where every lateral slice is an $R_{1} \times R_{1}$ identity matrix
	\For{$k = 1, \cdots, N-1$}
		\State $\tensor{G}^{\ne N}_S \leftarrow \tensor{G}^{\ne N}_S \boxast_2 (\tensor{G}_{k})_{S}$
	\EndFor
	\State $\mat{X}_{S[N]}^{new} \leftarrow \textsc{Mode-n-Unfolding}(\tensor{X}^{new}(\texttt{idxs}(:,1), \cdots ,\texttt{idxs}(:,N-1),:))$
	\State $\mat{G}_{N(2)}^{new} \leftarrow \mat{X}_{S[N]}^{new} \left( \left( \mat{G}_{S[2]}^{\ne N} \right)^\intercal \right)^\dagger$
	\State $\mat{G}_{N(2)} \leftarrow \begin{bmatrix} \mat{G}_{N(2)}^{old} \\ \mat{G}_{N(2)}^{new} \end{bmatrix}$ and reshape $\mat{G}_{N(2)}$ to $\tensor{G}_{N}$
	\State $\texttt{idxs}(:,N) \leftarrow \textsc{Randsample}(t^{new}, m)$
        \Comment $t^{new}$ is the temporal dimension of $\tensor{X}^{new}$
	\State $(\tensor{G}_{N}^{new})_{S} \leftarrow \tensor{G}_{N}^{new}(:, \texttt{idxs}(:,N), :)$
	\Statex // \texttt{Update stage for non-temporal modes}
	\For{$n = 1, \cdots, N-1$}
	\State Let $\left(\tensor{G}^{\ne n}_{new}\right)_{S}$ be a tensor of size $R_{n+1} \times m \times R_{n+1}$, where every lateral slice is an $R_{n+1} \times R_{n+1}$ identity matrix
	\For{$k = n+1, \cdots, N, 1, \cdots, n-1$}
	\State $\left(\tensor{G}^{\ne n}_{new}\right)_{S} \leftarrow \left(\tensor{G}^{\ne n}_{new}\right)_{S} \boxast_2 (\tensor{G}_{k})_{S}$
        \Comment If $k=N$, $(\tensor{G}_{k})_{S}$ will be $(\tensor{G}_{N}^{new})_{S}$
	\EndFor
	\State $\mat{X}_{S[n]}^{new} \leftarrow \textsc{Mode-n-Unfolding}(\tensor{X}^{new}(\texttt{idxs}(:,1), \cdots, \texttt{idxs}(:,n-1),:,\texttt{idxs}(:,n+1),\cdots \texttt{idxs}(:,N)))$
	\State $\mat{P}_{n} \leftarrow \mat{P}_{n} + \mat{X}_{S[n]}^{new}  \left(\mat{G}^{\ne n}_{new} \right)_{S[2]}$
	\State $\mat{Q}_{n} \leftarrow \mat{Q}_{n} + \left(\mat{G}^{\ne n}_{new} \right)_{S[2]}^\intercal \left(\mat{G}^{\ne n}_{new} \right)_{S[2]}$
	\State $\mat{G}_{n(2)} \leftarrow \mat{P}_{n} \mat{Q}_{n}^{\dagger}$ and reshape $\mat{G}_{n(2)}$ to $\tensor{G}_{n}$
	\State $\texttt{idxs}(:,n) \leftarrow \textsc{Randsample}(I_n, m)$
	\State $(\tensor{G}_{n})_{S} \leftarrow \tensor{G}_{n}(:, \texttt{idxs}(:,n), :)$
	\EndFor
	\EndFor
\end{algorithmic}
\end{breakablealgorithm}

\begin{breakablealgorithm}\label{alg:rstr_lev}
\caption{Randomized streaming TR decomposition with leverage-based sampling (rSTR-L)}
\textbf{Input:} Initial tensor $\tensor{X}^{init}$, TR-ranks $R_1, \cdots, R_N$, new data tensor $\tensor{X}^{new}$ and sampling size $m$

\textbf{Output:} TR-cores $\{ \tensor{G}_n \in \bb{R}^{R_n \times I_n \times R_{n+1}} \}_{n=1}^N$
\begin{algorithmic}[1]
	\Statex // \texttt{Initialization stage}
	\State Using TR-ALS-Sampled or other algorithms to decompose $\tensor{X}^{init}$ into TR-cores $\tensor{G}_{1}, \cdots, \tensor{G}_{N}$ 
	\State Compute the leverage-based probability distribution $\vect{p}_k$ of $\mat{G}_{k(2)}$ for $k = 1, \cdots, N$
	\For{$k = 1, \cdots, N$}
		\State $\texttt{idxs}(:,k) \leftarrow \textsc{Randsample}(I_k, m, true, \vect{p}_k)$
		\State $(\tensor{G}_{k})_{S} \leftarrow \tensor{G}_{k}(:, \texttt{idxs}(:,k), :)$
	\EndFor
	\For{$n = 1, \cdots, N-1$}
	\State Let $\tensor{G}^{\ne n}_S$ be a tensor of size $R_{n+1} \times m \times R_{n+1}$, where every lateral slice is an $R_{n+1} \times R_{n+1}$ identity matrix
	\For{$k = n+1, \cdots, N, 1, \cdots, n-1$}
		\State $\tensor{G}^{\ne n}_S \leftarrow \tensor{G}^{\ne n}_S \boxast_2 (\tensor{G}_{k})_{S}$
	\EndFor
	\State $\mat{X}_{S[n]}^{init} \leftarrow \textsc{Mode-n-Unfolding}(\tensor{X}(\texttt{idxs}(:,1), \cdots, \texttt{idxs}(:,n-1),:,\texttt{idxs}(:,n+1),\cdots \texttt{idxs}(:,N)))$
	\EndFor
	\For{$n = 1, \cdots, N-1$}	
	\State $\mat{P}_{n} \leftarrow \mat{X}_{S[n]}^{init} \mat{G}^{\ne n}_{S[2]}$
	\State $\mat{Q}_{n} \leftarrow \left( \mat{G}^{\ne n}_{S[2]} \right)^\intercal \mat{G}^{\ne n}_{S[2]}$
	\EndFor
	\For{$\tau = 1, \cdots, t$ time steps}
	\Statex // \texttt{Update stage for temporal mode}
	\State Let $\tensor{G}^{\ne N}_S$ be a tensor of size $R_{1} \times m \times R_{1}$, where every lateral slice is an $R_{1} \times R_{1}$ identity matrix
	\For{$k = 1, \cdots, N-1$}
	\State $\tensor{G}^{\ne N}_S \leftarrow \tensor{G}^{\ne N}_S \boxast_2 (\tensor{G}_{k})_{S}$
	\EndFor
	\State $\mat{X}_{S[N]}^{new} \leftarrow \textsc{Mode-n-Unfolding}(\tensor{X}^{new}(\texttt{idxs}(:,1), \cdots ,\texttt{idxs}(:,N-1),:))$
	\State $\mat{G}_{N(2)}^{new} \leftarrow \mat{X}_{S[N]}^{new} \left( \left( \mat{G}_{S[2]}^{\ne N} \right)^\intercal \right)^\dagger$
	\State $\mat{G}_{N(2)} \leftarrow \begin{bmatrix} \mat{G}_{N(2)}^{old} \\ \mat{G}_{N(2)}^{new} \end{bmatrix}$ and reshape $\mat{G}_{N(2)}$ to $\tensor{G}_{N}$
	\State Compute $\vect{p}_N$ for $\mat{G}_{N(2)}^{new}$
	\State $\texttt{idxs}(:,N) \leftarrow \textsc{Randsample}(t^{new}, m, true, \vect{p}_N)$
        \Comment $t^{new}$ is the temporal dimension of $\tensor{X}^{new}$
	\State $(\tensor{G}_{N})_{S} \leftarrow \tensor{G}_{N}(:, \texttt{idxs}(:,N), :)$
	\Statex // \texttt{Update stage for non-temporal modes}
	\For{$n = 1, \cdots, N-1$}
	\State Let $\left(\tensor{G}^{\ne n}_{new}\right)_{S}$ be a tensor of size $R_{n+1} \times m \times R_{n+1}$, where every lateral slice is an $R_{n+1} \times R_{n+1}$ identity matrix
	\For{$k = n+1, \cdots, N, 1, \cdots, n-1$}
	\State $\left(\tensor{G}^{\ne n}_{new}\right)_{S} \leftarrow \left(\tensor{G}^{\ne n}_{new}\right)_{S} \boxast_2 (\tensor{G}_{k})_{S}$
        \Comment If $k=N$, $(\tensor{G}_{k})_{S}$ will be $(\tensor{G}_{N}^{new})_{S}$
	\EndFor
	\State $\mat{X}_{S[n]}^{new} \leftarrow \textsc{Mode-n-Unfolding}(\tensor{X}^{new}(\texttt{idxs}(:,1), \cdots, \texttt{idxs}(:,n-1),:,\texttt{idxs}(:,n+1),\cdots \texttt{idxs}(:,N)))$
	\State $\mat{P}_{n} \leftarrow \mat{P}_{n} + \mat{X}_{S[n]}^{new}  \left(\mat{G}^{\ne n}_{new} \right)_{S[2]}$
	\State $\mat{Q}_{n} \leftarrow \mat{Q}_{n} + \left(\mat{G}^{\ne n}_{new} \right)_{S[2]}^\intercal \left(\mat{G}^{\ne n}_{new} \right)_{S[2]}$
	\State $\mat{G}_{n(2)} \leftarrow \mat{P}_{n} \mat{Q}_{n}^{\dagger}$ and reshape $\mat{G}_{n(2)}$ to $\tensor{G}_{n}$
	\State Recompute $\vect{p}_n$ for updated $\mat{G}_{n(2)}$
	\State $\texttt{idxs}(:,n) \leftarrow \textsc{Randsample}(I_n, m, true, \vect{p}_n)$
	\State $(\tensor{G}_{n})_{S} \leftarrow \tensor{G}_{n}(:, \texttt{idxs}(:,n), :)$
	\EndFor
	\EndFor
\end{algorithmic}
\end{breakablealgorithm}

\begin{breakablealgorithm}\label{alg:rstr_ksrft}
\label{SUPP:alg:rSTR-K}
\caption{Randomized streaming TR decomposition with KSRFT (rSTR-K)}
\textbf{Input:} Initial tensor $\tensor{X}^{init}$, TR-ranks $R_1, \cdots, R_N$, new data tensor $\tensor{X}^{new}$ and sketch size $m$

\textbf{Output:} TR-cores $\{ \tensor{G}_n \in \bb{R}^{R_n \times I_n \times R_{n+1}} \}_{n=1}^N$
\begin{algorithmic}[1]
	\Statex // \texttt{Initialization stage}
	\State Using TR-KSRFT-ALS or other algorithms to decompose $\tensor{X}^{init}$ into TR-cores $\tensor{G}_{1}, \cdots, \tensor{G}_{N}$ 
	\State Define random sign-flip operators $\mat{D}_j$ and FFT matrices $\mat{F}_j$, for $j \in [N]$
	\State Mix TR-cores: $\hat{\tensor{G}}_{n} \leftarrow \tensor{G}_{n} \times_2 (\mat{F}_n \mat{D}_n)$, for $ n = {1,2,\cdots, N}$ 
	\State Mix initial tensor: $\hat{\tensor{X}}^{init} \leftarrow \tensor{X} \times_1 (\mat{F}_1 \mat{D}_1) \times_2 (\mat{F}_2 \mat{D}_2) \cdots \times_N (\mat{F}_N \mat{D}_N)$
	\For{$k = 1, \cdots, N$}
		\State $\texttt{idxs}(:,k) \leftarrow \textsc{Randsample}(I_k, m)$
		\State $(\hat{\tensor{G}}_{k})_{S} \leftarrow \hat{\tensor{G}}_{k}(:, \texttt{idxs}(:,k), :)$
	\EndFor
	\For{$n = 1, \cdots, N-1$}
	\State Let $\hat{\tensor{G}}^{\ne n}_S$ be a tensor of size $R_{n+1} \times m \times R_{n+1}$, where every lateral slice is an $R_{n+1} \times R_{n+1}$ identity matrix
	\For{$k = n+1, \cdots, N, 1, \cdots, n-1$}
		\State $\hat{\tensor{G}}^{\ne n}_S \leftarrow \hat{\tensor{G}}^{\ne n}_S \boxast_2 (\hat{\tensor{G}}_{k})_{S}$
	\EndFor
	\State $\hat{\mat{X}}_{S[n]}^{init} \leftarrow \textsc{Mode-n-Unfolding}(\hat{\tensor{X}}(\texttt{idxs}(:,1), \cdots, \texttt{idxs}(:,n-1),:,\texttt{idxs}(:,n+1),\cdots \texttt{idxs}(:,N)))$
	\State $\hat{\mat{X}}_{S[n]}^{init} \leftarrow  \mat{D}_n \mat{F}_n^* \hat{\mat{X}}_{S[n]}$
	\EndFor
	\For{$n = 1, \cdots, N-1$}	
	\State $\mat{P}_{n} \leftarrow \hat{\mat{X}}_{S[n]}^{init} \hat{\mat{G}}^{\ne n}_{S[2]}$
	\State $\mat{Q}_{n} \leftarrow \left( \hat{\mat{G}}^{\ne n}_{S[2]} \right)^\intercal \hat{\mat{G}}^{\ne n}_{S[2]}$
	\EndFor
	\For{$\tau = 1, \cdots, t$ time steps}
	\State Redefine random sign-flip operators $\mat{D}_j$ and FFT matrices $\mat{F}_j$, for $j \in [N]$
	\State Mix new tensor: $\hat{\tensor{X}}^{new} \leftarrow \tensor{X} \times_1 (\mat{F}_1 \mat{D}_1) \times_2 (\mat{F}_2 \mat{D}_2) \cdots \times_N (\mat{F}_N \mat{D}_N)$ \label{SUPP:line:update-mix}
        \State Mix TR-cores: $\hat{\tensor{G}}_{n} \leftarrow \tensor{G}_{n} \times_2 (\mat{F}_n \mat{D}_n)$, for $ n = {1,2,\cdots, N-1}$
	\Statex // \texttt{Update stage for temporal mode}
	\State Let $\hat{\tensor{G}}^{\ne N}_S$ be a tensor of size $R_{1} \times m \times R_{1}$, where every lateral slice is an $R_{1} \times R_{1}$ identity matrix
	\For{$k = 1, \cdots, N-1$}
		\State $\hat{\tensor{G}}^{\ne N}_S \leftarrow \hat{\tensor{G}}^{\ne N}_S \boxast_2 (\hat{\tensor{G}}_{k})_{S}$
	\EndFor
	\State $\hat{\mat{X}}_{S[N]}^{new} \leftarrow \textsc{Mode-n-Unfolding}(\hat{\tensor{X}}^{new}(\texttt{idxs}(:,1), \cdots ,\texttt{idxs}(:,N-1),:))$
	\State $\hat{\mat{X}}_{S[N]}^{init} \leftarrow  \mat{D}_N \mat{F}_N^* \hat{\mat{X}}_{S[N]}$
	\State $\mat{G}_{N(2)}^{new} \leftarrow \Re\left(\hat{\mat{X}}_{S[N]}^{new} \right)\left( \Re\left( \left( \hat{\mat{G}}_{S[2]}^{\ne N} \right)^\intercal \right)\right)^\dagger$
	\State $\mat{G}_{N(2)} \leftarrow \begin{bmatrix} \mat{G}_{N(2)}^{old} \\ \mat{G}_{N(2)}^{new} \end{bmatrix}$ and reshape $\mat{G}_{N(2)}$ to $\tensor{G}_{N}$
	\State $\hat{\tensor{G}}_{N}^{new} \leftarrow \tensor{G}_{N}^{new} \times_2 (\mat{F}_N \mat{D}_N)$
	\State $\texttt{idxs}(:,N) \leftarrow \textsc{Randsample}(t^{new}, m)$
        \Comment $t^{new}$ is the temporal dimension of $\tensor{X}^{new}$
	\State $\left(\hat{\tensor{G}}_{N}^{new}\right)_{S} \leftarrow \hat{\tensor{G}}_{N}^{new}(:, \texttt{idxs}(:,N), :)$
	\Statex // \texttt{Update stage for non-temporal modes}
	\For{$n = 1, \cdots, N-1$}
	\State Let $\left(\hat{\tensor{G}}^{\ne n}_{new}\right)_{S}$ be a tensor of size $R_{n+1} \times m \times R_{n+1}$, where every lateral slice is an $R_{n+1} \times R_{n+1}$ identity matrix
	\For{$k = n+1, \cdots, N, 1, \cdots, n-1$}
	\State $\left(\hat{\tensor{G}}^{\ne n}_{new}\right)_{S} \leftarrow \left(\hat{\tensor{G}}^{\ne n}_{new}\right)_{S} \boxast_2 (\hat{\tensor{G}}_{k})_{S}$
        \Comment If $k=N$, $(\hat{\tensor{G}}_{k})_{S}$ will be $(\hat{\tensor{G}}_{N}^{new})_{S}$
	\EndFor
	\State $\hat{\mat{X}}_{S[n]}^{new} \leftarrow \textsc{Mode-n-Unfolding}(\hat{\tensor{X}}^{new}(\texttt{idxs}(:,1), \cdots, \texttt{idxs}(:,n-1),:,\texttt{idxs}(:,n+1),\cdots \texttt{idxs}(:,N)))$
	\State $\hat{\mat{X}}_{S[n]}^{init} \leftarrow  \mat{D}_n \mat{F}_n^* \hat{\mat{X}}_{S[n]}$
	\State $\mat{P}_{n} \leftarrow \mat{P}_{n} + \mat{X}_{S[n]}^{new}  \overline{\left(\mat{G}^{\ne n}_{new} \right)_{S[2]}}$
	\State $\mat{Q}_{n} \leftarrow \mat{Q}_{n} + \left(\mat{G}^{\ne n}_{new} \right)_{S[2]}^\intercal \overline{\left(\mat{G}^{\ne n}_{new} \right)_{S[2]}}$
	\State $\mat{G}_{n(2)} \leftarrow \Re(\mat{P}_{n}) \Re(\mat{Q}_{n})^{\dagger}$ and reshape $\mat{G}_{n(2)}$ to $\tensor{G}_{n}$
	\State $\hat{\tensor{G}}_{n} \leftarrow \tensor{G}_{n} \times_2 (\mat{F}_n \mat{D}_n)$
	\State $\texttt{idxs}(:,n) \leftarrow \textsc{Randsample}(I_n, m)$
	\State $(\hat{\tensor{G}}_{n})_{S} \leftarrow \hat{\tensor{G}}_{n}(:, \texttt{idxs}(:,n), :)$
	\EndFor
	\EndFor
\end{algorithmic}
\end{breakablealgorithm}

\end{sloppypar}
\end{document}